\documentclass[12pt,reqno]{amsart}
\usepackage{mathrsfs}
\usepackage{amssymb,amsmath,amsfonts,amsthm}
\usepackage[numbers,sort&compress]{natbib}
\usepackage{graphicx}
\usepackage{bm}
\usepackage[all,cmtip,line]{xy}
\usepackage{array}
\usepackage{color}
\numberwithin{equation}{section}

\newcommand{\dt}{\partial_t}

\newcommand{\dz}[0]{\partial_z}

\newcommand{\dvh}[0]{\mathrm{div}_{h}\,}
\newcommand{\nablah}[0]{\nabla_{h}}
\newcommand{\deltah}[0]{\Delta_{h}}

\newcommand{\subeqref}[2]{$\eqref{#1}_{#2}$}

\newtheorem{theorem}{Theorem}[section]
\newtheorem{corollary}[theorem]{Corollary}

\newtheorem{proposition}[theorem]{Proposition}

\everymath{\displaystyle}

\allowdisplaybreaks
\begin{document}

\title[ ]
{Local well-posedness of strong solutions to the non-isentropic compressible primitive equations with vertical diffusion}

\author{Rupert Klein}
\address{Rupert Klein, Institut f\"ur Mathematik, Freie Universit\"at Berlin, Berlin, Germany
}
\email{rupert.klein@math.fu-berlin.de}

\author{Jinkai Li}
\address{Jinkai Li, South China Research Center for Applied Mathematics and Interdisciplinary Studies, School of Mathematical Sciences, South China Normal University, Zhong Shan Avenue West 55, Guangzhou 510631, P. R. China}
\email{jklimath@m.scnu.edu.cn; jklimath@gmail.com}

\author{Xin Liu}
\address{Xin Liu, Department of Mathematics, Texas A\&M University, College Station, TX 77843-3368, USA}
\email{xliu23@tamu.edu}

	\author{Edriss S. Titi}
\address{Edriss S. Titi, Department of Mathematics, Texas A\&M University, College Station, TX 77840, USA.
Department of Applied Mathematics and Theoretical Physics, University of Cambridge,
Cambridge CB3 0WA, UK. Department of Computer Science and Applied Mathematics,
Weizmann Institute of Science, Rehovot 76100, Israel.}
\email{titi@math.tamu.edu; Edriss.Titi@damtp.cam.ac.uk}

\date{
Feb 23, 2026
}

\begin{abstract}
Due to the absence of dynamical equation in the vertical momentum component of the primitive equations (PEs) of atmospheric dynamics, the
vertical component of the velocity can be recovered only from the information on the other physical quantities, while utilizing the hydrostatic balance.
This causes one spatial derivative loss while leads to a stronger nonlinearity, comparing to the classic compressible Navier-Stokes equations.
As a result, the mathematical analysis on the compressible primitive equations is mathematically more challenging than that on the compressible Navier-Stokes
equations. In this paper, we consider the initial-boundary value problem to the non-isentropic compressible primitive equations with only vertical diffusion for the temperature, but without gravity. Local existence and uniqueness as well as the
continuous dependence on the initial data of strong
solutions are established for any suitably regular initial data. The initial velocity and pressure are assumed here to be of one order derivative higher regularity than that of the initial density.
 \end{abstract}



\maketitle

\allowdisplaybreaks
\section{Introduction}

The compressible primitive equations, formally obtained as the hydrostatic approximation of the compressible Navier-Stokes equations by replacing the vertical momentum balance equation with the hydrostatic balance
equation, are the fundamental equations for modern meteorological study (see, e.g., \cite[Chapter 4]{Richardson1965}). In particular, it is the starting point of many large scale
models in the theoretical investigations and practical weather predictions (see, e.g., \cite{Lions1992,JLLions1992,Washington2005}). In fact, such an approximation is reliable and useful in the following sense: (i)
the vertical scale of the atmosphere is significantly smaller than the planetary horizontal scale; (ii) the balance of gravity and pressure dominates the dynamic in the vertical direction; (iii) the vertical velocity is
usually hard to observe in reality. Rigorous justifications of the hydrostatic approximation of hydrodynamic equations can be found in \cite{Li2017,Azerad2001} (see also \cite{FugigaHHK,FuGiHiHuKaW1,FuruGigaKas,LiTitiYuan}) for incompressible flows and \cite{Liu2019} for
isentropic compressible flows.

In this work, our goal is to investigate the fundamental problem of the well-posedness of solutions to the compressible primitive equations with non-trivial entropy. In particular, we consider the non-isentropic compressible primitive equations with only vertical diffusion for the temperature:
\begin{equation}\label{eq:CPE}
	\begin{cases}
		\dt \rho + \dvh (\rho v)  + \dz (\rho w) = 0, \\
		\dt (\rho v) + \dvh(\rho v\otimes v) + \dz(\rho w v) + \nablah p = \dvh \mathbb S_h + \mu \partial_{zz} v, \\
		\dz p = 0, \\
		\dfrac{1}{\gamma-1}\bigl(\dt p + \dvh(p v) + \dz (pw) \bigr) + p(\dvh v + \dz w) = \dz (\kappa \dz \theta )+Q(\nabla v),
	\end{cases}
\end{equation}
with
\begin{eqnarray}
	&\mathbb S_h = \mu (\nablah v + \nablah^\top v) + \lambda \dvh v \mathbb I_2, \quad
	p = R \rho \theta, \nonumber\\
  &Q(\nabla v)=\mathbb S_h:\nabla_hv+\mu|\partial_zv|^2,\label{Q}
\end{eqnarray}
with constant coefficients $\gamma$, $\mu$, $\lambda$, $\kappa,$ and $R$ satisfying $\gamma > 1$, $\mu+\lambda>0$, $\mu>0$, $\kappa>0$, and $R>0$.
Here, $ \rho, v, w $, and $ p $ represent the density, the horizontal velocity, vertical velocity, and pressure, respectively. $ \dvh, \nablah $, and $\deltah $ are, and will be, the divergence, gradient, and Laplace operators in the horizontal variables $ (x,y) $, respectively.
We investigate system \eqref{eq:CPE} in the horizontally periodic channel
$$ \mathcal O:= \mathbb T^2 \times (0,1) = \lbrace (x,y,z) | (x,y) \in \mathbb T^2, z \in (0,1) \rbrace,
$$
and consider the following
boundary conditions:
\begin{equation}\label{bc:CPE}
	\dz \theta\big|_{z=0,1} = 0, \quad \dz v\big|_{z=0,1} = 0, \quad w\big|_{z=0,1} = 0.
\end{equation}

The isentropic compressible primitive equations have been investigated by the last two authors in \cite{LT2018a} for local strong solutions and \cite{LT2018b} for global weak solutions. The existence of global weak solutions is also studied by \cite{Wang2020}, independently. See also, \cite{Jiu2018,Ersoy2011a,Ersoy2012,Gatapov2005}.

On the other hand, the incompressible primitive equations have been the subject of intensive mathematical research since the introduction
by Lions, Temam, and Wang in \cite{Lions1992,JLLions1992}. For instance, Guill\'en-Gonz\'alez, Masmoudi, and Rodr\'iguez-Bellido in
\cite{GuillenGonzalez2001} study the local existence of strong solutions and global existence of strong solutions with small initial data. In
\cite{HuTemamZiane2003}, the authors address the global existence of strong solutions in a domain with small depth. The first breakthrough
concerning the global well-posedness of incompressible primitive equations is obtained by Cao and Titi in \cite{Cao2007}. See also,
\cite{Cao2014a,Cao2014b,Cao2016,Cao2016a,CaoLiTiti2020,CaoLiTitiWang2024,Boling,MHATK,MHTKa,GMK,IKMZ,Li2017a,LiYuan2022} and the references therein for
related literatures; in particular, global well-posedness of strong solutions was proved in \cite{Cao2016,Cao2016a,CaoLiTiti2020,CaoLiTitiWang2024} to the primitive equations with only horizontal viscosity.

We have introduced the PE diagram in \cite{LT2018LowMach1}, which concerns the low Mach number limit and the small aspect ratio (between the vertical and horizontal scales) limit. In \cite{LT2018LowMach1,LT2018LowMach2}, we also establish the low Mach number limit for the isentropic compressible primitive equations. However, the counterpart study of the compressible Navier-Stokes-Fourier equations (see, e.g., \cite{Alazard2006}) indicates that the PE diagram might be completely different for non-isentropic flows, due to the additional heat conductivity. We leave such a subject to future study.

For interested readers, we refer to \cite{Temam1977,Lions1996,Lions1998,Feireisl2004,Feireisl2009a} for the study of hydrodynamic equations.

One can easily see that, compared to the Navier-Stokes equations, the evolutionary equation for vertical velocity is missing in system
\eqref{eq:CPE}. In fact, this is one of the main challenges in the study of the compressible primitive equations. In order to have a better
understanding of the system, we derive a representation of vertical velocity in terms of horizontal velocity, density, and pressure.
Consequently, we will reformulate system \eqref{eq:CPE} to an equivalent one.

Use \subeqref{eq:CPE}{1} to rewrite \subeqref{eq:CPE}{2} as
$$
\rho(\partial_tv+(v\cdot\nabla_h)v+w\partial_zv)+\nabla_hp=\mu\Delta v+(\mu+\lambda)\nabla_h\text{div}_hv
$$
and define
\begin{equation}\label{def:rho-1}
	\sigma := \frac1\rho.
\end{equation}
Then, it is clear that
\begin{equation}
\partial_tv+(v\cdot\nabla_h)v+w\partial_zv+\sigma \nabla_hp=\mu\sigma\Delta v+(\mu+\lambda)\sigma \nabla_h\text{div}_hv. \label{EQv00}
\end{equation}

Note that \subeqref{eq:CPE}{3} implies that $p$ is independent of the vertical variable $z$ and one can use \subeqref{eq:CPE}{3} to rewrite \subeqref{eq:CPE}{4} as
\begin{equation*}
 \partial_tp+v\cdot\nabla_hp+\gamma p(\text{div}_hv+\partial_zw)=(\gamma-1)\kappa\partial_z^2\theta+(\gamma-1)Q(\nabla v).
\end{equation*}
Separating the $z$-average part and the fluctuation part of the above equation, and recalling the boundary conditions $\partial_z\theta|_{z=0,1}=w|_{z=0,1}=0$, one obtains that
\begin{eqnarray}
 \partial_tp+\bar v\cdot\nabla_hp+\gamma p\text{div}_h\bar v=(\gamma-1)\overline{Q(\nabla v)},\label{EQp00}\\
\tilde v\cdot\nabla_hp+\gamma p(\text{div}_h\tilde v+\partial_zw)=(\gamma-1)\kappa\partial_z^2\theta+(\gamma-1)\widetilde{Q(\nabla v)}. \label{EQw000}
\end{eqnarray}
Here, for a function $f$, we use
\begin{equation}\label{TBf}
\bar f:=\int_0^1f dz\quad\mbox{and}\quad  \tilde f:=f-\bar f,
\end{equation}
to represent the vertical integral (equivalently, average) and the vertical fluctuation of the quantity considered, respectively.

Since $\theta = R^{-1} \rho^{-1} p=R^{-1}\sigma p,$
it follows from \eqref{EQw000} and \subeqref{eq:CPE}{3} that
\begin{eqnarray}\label{eq:dz-w}
		 \dz w &=&-\text{div}_h\tilde v+\frac{\gamma-1}{\gamma p}\left(\kappa\partial_z^2\theta+\widetilde{Q(\nabla v)}\right)-\frac{1}{\gamma p}\tilde v\cdot\nabla_hp\nonumber\\
&=&\dfrac{(\gamma-1)\kappa}{\gamma R}\partial_{zz} \sigma -\text{div}_h\tilde v-\frac{1}{\gamma p}\left(\tilde v\cdot\nabla_hp-(\gamma-1)\widetilde{Q(\nabla v)}\right)\nonumber\\
&=&\nu\partial_z^2\sigma-\phi(v,p),\label{EQw00}
\end{eqnarray}
where $\nu=\dfrac{(\gamma-1)\kappa}{\gamma R}$ and
\begin{eqnarray}
  \phi(v,p):=\text{div}_h\tilde v+\frac{1}{\gamma p}\left(\tilde v\cdot\nabla_hp-(\gamma-1)\widetilde{Q(\nabla v)}\right).\label{phi}
\end{eqnarray}
Recalling that $p$ is independent of $z$, it holds that $\dz\sigma=R\dz(p^{-1}\theta)= R p^{-1}\dz\theta$. Then, the boundary condition \eqref{bc:CPE} implies
\begin{equation}\label{bc:rf-CPE}
\dz \sigma\big|_{z=0,1} = 0.
\end{equation}
Thanks to this and recalling that $w|_{z=0}$, it follows from \eqref{EQw00} that
\begin{equation}
  w=\nu\partial_z\sigma-\int_0^z\phi(v,p)dz'.\label{EQw00'}
\end{equation}

It follows from \subeqref{eq:CPE}{1} that
\begin{equation*}\label{eq:rho-1}
	\dt \rho^{-1} + v \cdot \nablah \rho^{-1} - \rho^{-1} \dvh v = \rho^{-2} \dz (\rho w) = \rho^{-1} \dz w - w \dz \rho^{-1},
\end{equation*}
from which, using \eqref{eq:dz-w} and recalling $\sigma=\frac1\rho$, one gets
\begin{equation}\label{EQs00}
	\partial_t\sigma+v\cdot\nabla_h\sigma+w\partial_z\sigma+\sigma(\phi(v,p)-\text{div}_hv)=\nu\sigma\partial_z^2\sigma.
\end{equation}

Now, collecting \eqref{EQv00}, \eqref{EQp00}, \eqref{EQw00'}, \eqref{EQs00}, we have the following reformulated system which is equivalent to \eqref{eq:CPE}
\begin{eqnarray}
  &\partial_tv+(v\cdot\nabla_h)v+w\partial_zv+\sigma\nabla_hp=\mu\sigma\Delta v+(\mu+\lambda)\sigma\nabla_h\text{div}_hv,\label{EQv0}\\
  &w=\nu\partial_z\sigma-\int_0^z\phi(v,p)dz',\label{EQw0}\\
  &\partial_t\sigma+v\cdot\nabla_h\sigma+w\partial_z\sigma+\sigma(\phi(v,p)-\text{div}_hv)=\nu\sigma\partial_z^2\sigma,
  \label{EQs0}\\
  &\partial_tp+\bar v\cdot\nabla_hp+\gamma p\text{div}_h\bar v=(\gamma-1)\overline{Q(\nabla v)}, \label{EQp0}
\end{eqnarray}
where $\phi(v,p)$ and $Q(\nabla v)$ are expressed as in (\ref{phi}) and (\ref{Q}), respectively. The boundary conditions read as
\begin{eqnarray}
  &v, \sigma, \mbox{ and } p \mbox{ are periodic in }x, y, \label{BC1}\\
  &\partial_zv|_{z=0,1}=0, \quad\partial_z\sigma|_{z=0,1}=0. \label{BC2}
\end{eqnarray}
The initial condition is
\begin{equation}
  \label{IC}
  (v, \sigma, p)|_{t=0}=(v_{0}, \sigma_{0}, p_{0}).
\end{equation}

Before stating the main result, we recall some standard notations. For positive integer $k$ and $q\in[1,\infty]$, $L^q(\mathcal O)$ and
$L^q(\mathbb T^2)$ are the Lebesgue spaces,
$W^{k,q}(\mathcal O)$ and $W^{k,q}(\mathbb T^2)$ are the Sobolev spaces. If $q=2$, we use $H^k$ instead of $W^{k,2}$. For simplicity,
we will use $\|\cdot\|_q$ to denote $L^q(\mathcal O)$ or $L^q(\mathbb T^2)$ norms of the corresponding function,
if the meaning is clear from the context.

We are now in the position to state the main result of this paper.

\begin{theorem}\label{thmmain}
Given $v_{0}\in H^3(\mathcal O)$, $\sigma_{0} \in H^2(\mathcal O)$, and $p_{0}\in H^3(\mathbb T^2),$ such that
$$
\sigma_{0}\geq\underline\sigma, \quad p_{0}\geq\underline p,\quad \partial_zv_{0}|_{z=0,1}=0,\quad \partial_z\sigma_{0}|_{z=0,1}=0,
$$
for two positive numbers $\underline\sigma$ and $\underline p$.
Then, there is a positive time $\mathcal T_0$,
depending only on $\gamma$, $\nu$, $\mu$, $\lambda$, $\underline\sigma$, $\underline p$, and
  $\|v_{0}\|_{H^3}^2+\|\sigma_{0}\|_{H^2}^2+\|p_0\|_{H^3}^2$, such that system (\ref{EQv0})--(\ref{EQp0}),
  subject to (\ref{BC1})--(\ref{IC}), has a unique local solution $(v,\sigma,p)$ on $\mathcal O\times(0,\mathcal T_0)$,
  satisfying
  \begin{eqnarray*}
  &&\inf_{(x,y,z,t)\in\mathcal O\times[0,\mathcal T_0]}\sigma\geq0.5\underline\sigma,\quad \inf_{(x,y,t)\in\mathbb T^2\times[0,\mathcal T_0]}p\geq0.5\underline p,\\
    &&v\in C([0,\mathcal T_0]; H^2(\mathcal O))\cap L^\infty(0,\mathcal T_0; H^3(\mathcal O))\cap L^2(0,\mathcal T_0; H^4(\mathcal O)),\\
    &&\sigma\in C([0,\mathcal T_0]; H^1(\mathcal O))\cap L^\infty(0,\mathcal T_0; H^2(\mathcal O)),\quad \partial_z\sigma\in L^2(0,\mathcal T_0; H^2(\mathcal O)),\\
    &&\partial_tv\in L^2(0,\mathcal T_0; H^2(\mathcal O)),\quad \partial_t\sigma\in L^2(0,\mathcal T_0; H^1(\mathcal O)),\\
    &&p\in C([0,\mathcal T_0]; H^2(\mathbb T^2))\cap L^\infty(0,\mathcal T_0; H^3(\mathbb T^2)),\quad\partial_tp\in  L^2(0,\mathcal T_0; H^2(\mathbb T^2)).
  \end{eqnarray*}
\end{theorem}

Some comments on the proof of Theorem \ref{thmmain} are presented as follows.
Regularity assumption $\sigma_0\in H^2(\mathcal O)$ comes from the following observation: for the toy model
$$
\partial_t\sigma=\nu\sigma\partial_z^2\sigma,
$$
in order that $H^k(\mathcal O)$ regularity propagates, one needs $k\geq2$.
Fortunately, such $H^2(\mathcal O)$ regularity on $\sigma_0$ is still sufficient for the more complicated equation \eqref{EQs0}.
While for the aim of getting $H^2(\mathcal O)$ estimate on $\sigma$, one may apply $\Delta$ to \eqref{EQs0} and, as a result, the
quantity $\Delta\phi(v,p)$ and further, recalling the expression of $\phi(v,p)$, $\nabla_h\Delta_hp$ will be encountered.
In order to get information on $\nabla_h\Delta_hp$, by equation \eqref{EQp0}, one requires $p_0\in H^3(\mathbb T^2)$ and moreover
the quantity $\nabla_h^4v$ will be encountered. Due to this and by \eqref{EQv0}, one has to assume $v_0\in H^3(\mathcal O)$.
The existence of solutions to system \eqref{EQv0}--\eqref{EQp0}, subject to \eqref{BC1}--\eqref{IC}, is established by parabolic
regularization argument. A regularized system is introduced by adding the regularizing dissipative terms $\epsilon\Delta_h\sigma$ and $\epsilon\Delta_hp$
to \eqref{EQs0} and \eqref{EQp0}, respectively, see \eqref{EQv}--\eqref{EQp}, below. Such a regularized system is a semi-linear
parabolic system. The main effort of the proof is then to carry out
suitable $\epsilon$-independent {\it a priori} estimates and the local existence
follows passing to the limit $\epsilon\rightarrow0$.
Continuous dependence on the initial data, which in particular implies
the uniqueness, is proved by performing energy estimates to the difference between two solutions. We employ $L^2$ energy estimate to the
subtracted $\sigma$ equation which contains a term involving $\nabla_hp$. This, due to the absence of dissipation
in the pressure equation, forces us to perform $H^1(\mathbb T^2)$ type energy estimate to the subtracted pressure equation and
further $H^1(\mathcal O)$ energy to the subtracted velocity equations.

The rest of this paper is arranged as follows: in the next section, Section \ref{sec3}, we consider the $\epsilon$-regularized system
and perform suitable a priori estimates; the main result, Theorem \ref{thmmain}, is proved in Section \ref{sec4}; some commutator estimates are proved in Appendix A; local existence and uniqueness of the $\epsilon$-regularized system is proved in Appendix B.


\section{The $\epsilon$-regularized system}
\label{sec3}
\subsection{Some calculus inequalities}
As basic tools, some calculus inequalities are stated in this subsection. Let $\Omega$ be a bounded smooth domain in $\mathbb R^n$ with $n\geq2$. Given a positive integer $m$ and arbitrary multi-index $\alpha$ with $|\alpha|=m$. Let $1\leq q\leq\infty$. Recall the following standard calculus inequalities (see Appendix A for the proof):
\begin{equation}
  \|D^\alpha(fg)\|_q\leq C\left(\|f\|_{r_1}\|g\|_{W^{m, s_1}}+\|g\|_{r_2}\|f\|_{W^{m,s_2}}\right) \label{CAL}
\end{equation}
and
\begin{equation}
  \|[D^\alpha, f]g\|_q\leq C\left(\|\nabla f\|_{r_1}\|g\|_{W^{m-1, s_1}}+\|g\|_{r_2}\|f\|_{W^{m,s_2}}\right),\label{COME}
\end{equation}
for any functions $f$ and $g$ such that the quantities on the right-hand sides are finite, where
\begin{equation}
  [D^\alpha, f]g : = D^\alpha(fg) - f D^\alpha g,
\end{equation}
$C$ is a positive constant depending only on $m, q, r_1, r_2, s_1, s_2,n,\Omega$, and
\begin{equation}\label{CDT}
\frac1q=\frac{1}{r_1}+\frac{1}{s_1}=\frac{1}{r_2}+\frac{1}{s_2}, \quad 1\leq q, r_1, r_2, s_1, s_2\leq\infty.
\end{equation}
As direct consequences,
for any $q$, $r_1$, $r_2$, $s_1$, $s_2$ satisfying (\ref{CDT}), it holds that
$$
\|fg\|_{W^{m,q}}\leq C\left(\|f\|_{r_1}\|g\|_{W^{m, s_1}}+\|g\|_{r_2}\|f\|_{W^{m,s_2}}\right),
$$
where $C$ is a positive constant depending only on $m, q, r_1, r_2, s_1, s_2,n,\Omega$. In particular, one gets
\begin{equation}\label{ALGmq}
  \|fg\|_{W^{m,q}}\leq C\left(\|f\|_{\infty}\|g\|_{W^{m,q}}+\|g\|_{\infty}\|f\|_{W^{m,q}}\right)
\end{equation}
and thus
\begin{equation*}
   \|fg\|_{W^{m,q}\cap L^\infty}\leq C \|f\|_{W^{m,q}\cap L^\infty}\|g\|_{W^{m,q}\cap L^\infty},
\end{equation*}
for a positive constant $C$ depending only on $m, q, n,$ and $\Omega$. Thanks to this and by the Sobolev embedding $W^{m,q}\hookrightarrow
L^\infty$ for $mq>n$, one has
\begin{equation*}
  \|fg\|_{W^{m,q}}\leq C \|f\|_{W^{m,q}}\|g\|_{W^{m,q}},\quad\mbox{if }mq>n,
\end{equation*}
for a positive constant $C$ depending only on $m, q, n,$ and $\Omega$. Applying this to the case that $n=2,3$, one gets for any integer $k\geq2$ hat
\begin{equation}
\label{ALGHk}
  \|fg\|_{H^k(V)}\leq C\|f\|_{H^k(V)}\|g\|_{H^k(V)},
\end{equation}
for a positive constant $C$ depending only on $k$ and $V$, where $V\subseteq\mathbb R^n (n=2,3)$ is a bounded smooth domain.

\subsection{The regularized system and local well-posedness}
Given $\epsilon>0$, consider the following regularized system
\begin{eqnarray}
  &\partial_tv+(v\cdot\nabla_h)v+w\partial_zv+\sigma\nabla_hp=\mu\sigma\Delta v+(\mu+\lambda)\sigma\nabla_h\text{div}_hv,\label{EQv}\\
  &w=\nu\partial_zv-\int_0^z\phi(v,p)dz',\label{EQw}\\
  &\partial_t\sigma+v\cdot\nabla_h\sigma+w\partial_z\sigma+\sigma(\phi(v,p)-\text{div}_hv)=\nu\sigma\partial_z^2\sigma+\epsilon\Delta_h\sigma,
  \label{EQs}\\
  &\partial_tp+\bar v\cdot\nabla_hp+\gamma p\text{div}_h\bar v=\epsilon\Delta_hp+(\gamma-1)\overline{Q(\nabla v)}, \label{EQp}
\end{eqnarray}
where $\phi(v,p)$ and $Q(\nabla v)$ are expressed as in (\ref{phi}) and (\ref{Q}), respectively.

Assume that the initial data $(v_{0}, \sigma_{0}, p_{0})$ satisfies
\begin{eqnarray}
    (v_{0}, \sigma_{0})\in H^3(\mathcal O),\quad p_{0}\in H^3(\mathbb T^2),\quad
    \sigma_{0}\geq\underline\sigma, \quad p_{0}\geq\underline p,\label{ASSUMIC1}\\
    \partial_zv_{0}|_{z=0,1}=0,\quad \partial_z\sigma_{0}|_{z=0,1}=0.\label{ASSUMIC2}
\end{eqnarray}

For any initial data $(v_{0}, \sigma_{0}, p_{0})$ satisfying (\ref{ASSUMIC1})--(\ref{ASSUMIC2}), system (\ref{EQv})--(\ref{EQp}),
subject to (\ref{BC1})--(\ref{IC}), has a unique local solution, that is we have the following proposition.

\begin{proposition}
  \label{PROPLOC-e}
  Given $(v_{0}, \sigma_{0}, p_{0})$ satisfying (\ref{ASSUMIC1})--(\ref{ASSUMIC2}).
  Then, there is a positive time $T_\epsilon$,
  depending only on $\epsilon$, $\underline\sigma$, $\underline p$, and
  $\|(v_{0},\sigma_{0})\|_{H^3(\mathcal O)}+\|p_{0}\|_{H^3(\mathbb T^2)}$, such that system (\ref{EQv})--(\ref{EQp}),
  subject to (\ref{BC1})--(\ref{IC}), has a unique local in time solution $(v,\sigma,p)$ on $\mathcal O\times(0,T_\epsilon)$,
  satisfying
  \begin{eqnarray*}
  &&\inf_{(x,y,z,t)\in\mathcal O\times(0,T_\epsilon)}\sigma>0,\quad \inf_{(x,y,t)\in\mathbb T^2\times(0,T_\epsilon)}p>0,\\
    &&(v,\sigma)\in C([0,T_\epsilon]; H^3(\mathcal O))\cap L^2(0,T_\epsilon; H^4(\mathcal O)),\quad(\partial_tv,\partial_t\sigma)\in L^2(0,T_\epsilon; H^2(\mathcal O)),\\
    &&p\in C([0,T_\epsilon]; H^3(\mathbb T^2))\cap L^2(0,T_\epsilon; H^4(\mathbb T^2)),\quad\partial_tp\in  L^2(0,T_\epsilon; H^2(\mathbb T^2)).
  \end{eqnarray*}
\end{proposition}

\begin{proof}
This is a direct corollary of Proposition \ref{PROPA3-1} in Appendix \ref{APPENDIXB}.
\end{proof}

\subsection{$\epsilon$-independent conditional energy estimates}

Denote $\Phi_1, \Phi_2,$ and $\Phi_2$ as
\begin{eqnarray}
  &&\Phi_1=\Phi_1(v,\sigma,p):=-(v\cdot\nabla_h)v-w\partial_zv-\sigma\nabla_hp,\label{Phi1}\\
  &&\Phi_2=\Phi_2(v,\sigma,p):=-w\partial_z\sigma+\sigma(\text{div}_hv-\phi(v,p)),\label{Phi2}\\
  &&\Phi_3=\Phi_3(v,p):=(\gamma-1)\overline{Q(\nabla v)}-\gamma p\text{div}_h\bar v, \label{Phi3}
\end{eqnarray}
with $w$, $\phi(v,p)$, and $Q(\nabla v)$ expressed as (\ref{EQw0}), (\ref{phi}), and \eqref{Q}, respectively. With these notations, system (\ref{EQv})--(\ref{EQp}) can be
rewritten equivalently as
\begin{eqnarray}
  \partial_tv-\mu\sigma\Delta v-(\mu+\lambda)\sigma\nabla_h\text{div}_hv&=&\Phi_1,\label{EQv'}\\
  \partial_t\sigma+v\cdot\nabla_h\sigma-\nu\sigma\partial_z^2\sigma-\epsilon\Delta_h\sigma&=&\Phi_2,
  \label{EQs'}\\
  \partial_tp+\bar v\cdot\nabla_hp-\epsilon\Delta_hp&=&\Phi_3. \label{EQp'}
\end{eqnarray}

Some estimates for $\Phi_i, i=1,2,3$, are given in the following proposition.

\begin{proposition}
\label{PROP-ESTPhis}
Given a positive time $\mathcal T$ and let $(v,\sigma, p)$ be a solution to system (\ref{EQv})--(\ref{EQp}), subject to (\ref{BC1})--(\ref{IC}), on $\mathcal O\times(0,\mathcal T)$, enjoying the regularities
stated in Proposition \ref{PROPLOC-e}. Assume that
\begin{eqnarray}
  \sup_{0\leq t\leq\mathcal T}\left(\|v\|_{H^3(\mathcal O)}^2+\|\sigma\|_{H^2(\mathcal O)}^2+\|p\|_{H^3(\mathbb T^2)}^2\right)(t)\leq M,  \label{ASSUM-APRM}\\
  \inf_{(x,y,z,t)\in\mathcal O\times(0,\mathcal T)}\sigma\geq 0.5\underline\sigma,\quad \inf_{(x,y,t)\in\mathcal T^2\times(0,\mathcal T)}
  p\geq 0.5\underline p,\label{ASSUM-LBsp}
\end{eqnarray}
for some positive number $M$.
Let $\Phi_1, \Phi_2,$ and $\Phi_3$ be given by (\ref{Phi1}),
(\ref{Phi2}), and (\ref{Phi3}), respectively.
Then, the following estimates hold
\begin{align*}
  &\|(\Phi_1, \Phi_2, \Phi_3)\|_1\leq C, \quad \|\Phi_1\|_{H^2}\leq \nu C_*\|\partial_zv\|_{H^2}\|\partial_z\sigma\|_{H^2}+C,
  \\
  &\|\Phi_2\|_{H^2}\leq C\left(1+\|\partial_z\sigma\|_{H^2}^\frac32\right), \quad
  \|\Phi_3\|_{H^3}\leq C(1+\|v\|_{H^4}),
\end{align*}
for any $t\in(0,\mathcal T)$,
where $C_*$ is an absolute positive constant, and $C$ is a positive constant depending only on $\gamma$, $\nu$, $\mu$,
$\lambda$, $\underline p$, and $M$.
\end{proposition}

\begin{proof}
Recalling the expression of $w$ in (\ref{EQw}), it follows from the H\"older, Sobolev, and Young inequalities and
(\ref{ASSUM-APRM})--(\ref{ASSUM-LBsp}) that
\begin{equation}
  \|\phi(v,p)\|_2\leq C(\|\nabla v\|_2+\|v\|_4\|\nabla_hp\|_4+\|\nabla v\|_4^2)\leq C(1+\|v\|_{H^1}^2+\|p\|_{H^1}^2)\leq C,
  \label{L2phi}
\end{equation}
for a positive constant $C$ depending only on $\gamma, \mu, \lambda, \underline p,$ and $M$. Thanks to this and recalling the expression of $w$
in (\ref{EQw}), it follows from (\ref{ASSUM-APRM}) that
\begin{equation}
  \label{L2w}
  \|w\|_2\leq\nu\|\partial_z\sigma\|_2+\left\|\int_0^z\phi(v,p)dz'\right\|_2\leq C(1+\|\sigma\|_{H^1}^2+\|\phi(v,p)\|_2)\leq C,
\end{equation}
for a positive constant $C$ depending only on $\gamma$, $\mu$, $\lambda$, $\nu$, $\underline p,$ and $M$.

Thanks to (\ref{L2phi}) and (\ref{L2w}),
it follows from the expressions of $\Phi_i, i=1, 2, 3,$ the Cauchy inequality, and (\ref{ASSUM-APRM}) that
\begin{align}
  \|(\Phi_1, \Phi_2, \Phi_3)\|_1\leq&\|v\|_2\|\nabla_hv\|_2+\|w\|_2\|\partial_zv\|_2+\|\sigma\|_2\|\nabla_hp\|_2\nonumber\\
  &+\|w\|_2\|\partial_z\sigma\|_2+\|\sigma\|_2(\|\text{div}_hv\|_2  +\|\phi(v,p)\|_2)\nonumber\\
  &+C(\|\nabla v\|_2^2+\|p\|_2\|\nabla_hv\|_2)\leq C,\label{L1Phis}
\end{align}
for a positive constant $C$ depending only on $\gamma$, $\mu$, $\lambda$, $\nu$, $\underline p,$ and $M$.

By direct calculations and using the Sobolev inequality, one has
\begin{eqnarray}
  \left\|\frac1p\right\|_{H^2}^2&\leq&C\int\left(\frac{1}{p^2}+\frac{|\nabla_hp|^2}{p^4}+\frac{|\nabla_h^2p|^2}{p^4}+\frac{|\nabla_hp|^4}{p^6}
  \right)dxdy\nonumber\\
  &\leq& C(1+\|\nabla_hp\|_4^4+\|\nabla_h^2p\|_2^2)\leq C(1+\|p\|_{H^2}^4)\leq C,\label{H21P}
\end{eqnarray}
for a positive constant depending only on $\underline p$ and $M$. Thanks to this,
recalling the expression of $\phi(v,p)$ in (\ref{phi}), and applying \eqref{ALGHk}, it follows from
(\ref{ASSUM-APRM}) that
\begin{equation}
  \|\phi(v,p)\|_{H^2}\leq C\left(\|\nabla v\|_{H^2}+\left\|\frac1p\right\|_{H^2}(\|v\|_{H^2}\|\nabla_hp\|_{H^2}+\|\nabla v\|_{H^2}^2)\right)
  \leq C,\label{H2phi}
\end{equation}
for a positive constant $C$ depending only on $\gamma$, $\mu$, $\lambda$, $\underline p$, and $M$. Then, recalling the expression of $w$ in (\ref{EQw}), it follows that
\begin{equation}
  \label{H1w}
  \|w\|_{H^1}\leq C(\|\partial_z\sigma\|_{H^1}+\|\phi(v,p)\|_{H^1}) \leq C
\end{equation}
and
\begin{equation}
  \label{H2w}
  \|w\|_{H^2}\leq\nu\|\partial_z\sigma\|_{H^2}+\left\|\int_0^z\phi(v,p)dz'\right\|_{H^2}\leq\nu\|\partial_z\sigma\|_{H^2}+C,
\end{equation}
for a positive constant $C$ depending only on $\gamma,$ $\nu$, $\mu$, $\lambda$, $\underline p$, and $M$.

Recalling the
expression of $\Phi_1$ in (\ref{Phi1}), applying (\ref{ALGHk}), and
recalling (\ref{ASSUM-APRM}) and (\ref{H2w}), one deduces
\begin{eqnarray}
  \|\Phi_1\|_{H^2}&\leq&\|(v\cdot\nabla_h)v\|_{H^2}+\|w\partial_zv\|_{H^2}+\|\sigma\nabla_hp\|_{H^2}\nonumber\\
  &\leq&C_*\|(v,\sigma)\|_{H^2}\|(\nabla_hv,\nabla_hp)\|_{H^2}+C_*\|w\|_{H^2}\|\partial_zv\|_{H^2}\nonumber\\
  &\leq&C+C_*\nu\|\partial_z\sigma\|_{H^2}\|\partial_zv\|_{H^2},\label{H2Phi1}
\end{eqnarray}
where $C_*$ is an absolute positive constant, and $C$ is a positive constant depending only on $\gamma$, $\mu$, $\lambda$, $\nu$,
$\underline p,$ and $M$.
Applying (\ref{ALGmq}) and (\ref{ALGHk}), it
follows from (\ref{ASSUM-APRM}), (\ref{H2phi}), (\ref{H1w}), (\ref{H2w}), and the Gagliardo-Nirenberg and Young inequalities that
\begin{eqnarray}
  \|\Phi_2\|_{H^2}&\leq&\|w\partial_z\sigma\|_{H^2}+\|\sigma\text{div}_hv\|_{H^2}+\|\sigma\phi(v,p)\|_{H^2}\nonumber\\
  &\leq&C[\|w\|_{H^2}\|\partial_z\sigma\|_\infty+\|w\|_\infty\|\partial_z\sigma\|_{H^2}+\|\sigma\|_{H^2}(\|\text{div}_hv\|_{H^2}+
  \|\phi(v,p)\|_{H^2})]\nonumber\\
  &\leq&C\left[(1+\|\partial_z\sigma\|_{H^2})\|\partial_z\sigma\|_{H^1}^\frac12\|\partial_z\sigma\|_{H^2}^\frac12+\|w\|_{H^1}^\frac12\|w\|_{H^2}^\frac12
  \|\partial_z\sigma\|_{H^2}+1\right]\nonumber\\
  &\leq&C\left[(1+\|\partial_z\sigma\|_{H^2})\|\partial_z\sigma\|_{H^2}^\frac12+(1+\|\partial_z\sigma\|_{H^2})^\frac12
  \|\partial_z\sigma\|_{H^2}+1\right]\nonumber\\
  &\leq& C\left(1+\|\partial_z\sigma\|_{H^2}^\frac32\right), \label{H2Phi2}
\end{eqnarray}
for a positive constant $C$ depending only on $\gamma$, $\mu$, $\lambda$, $\nu$, $\underline p,$ and $M$.
Applying (\ref{ALGmq}), it follows from the Sobolev inequality
and (\ref{ASSUM-APRM}) that
\begin{eqnarray}
  \|\Phi_3\|_{H^3}&\leq&C(\|\nabla v\|_{H^3}\|\nabla v\|_\infty+\|p\|_\infty\|\nabla v\|_{H^3}+\|p\|_{H^3}\|\nabla v\|_\infty)\nonumber\\
  &\leq&C(\|v\|_{H^3}\|v\|_{H^4}+\|p\|_{H^2}\|v\|_{H^4}+\|p\|_{H^3}\|v\|_{H^3})\nonumber\\
  &\leq& C(1+\|v\|_{H^4}), \label{H3Phi3}
\end{eqnarray}
for a positive constant $C$ depending only on $\gamma$, $\mu$, $\lambda,$ and $M$.
\end{proof}

\begin{proposition}
  \label{PROP-LBsp-UnderM}
Assume that all the assumptions stated in Proposition \ref{PROP-ESTPhis} hold. Then, it holds that
\begin{align*}
p\geq\underline pe^{-Ct},\quad \sigma\geq\underline\sigma e^{-Ct}, \quad\mbox{on }\mathcal O\times(0,\mathcal T),
\end{align*}
for a positive constant $C$ depending only on $\gamma, \mu, \lambda, \underline p,$ and $M$.
\end{proposition}

\begin{proof}
Applying the comparision principle to \eqref{EQs} and \eqref{EQp}, and by the Sobolev inequality, it follows from \eqref{ASSUM-APRM} and \eqref{H2phi} that
$$
  p(x,y,z,t)\geq\underline pe^{-\gamma\int_0^t\|\text{div}_h\bar v\|_\infty ds}\geq\underline pe^{- C\int_0^t\|v\|_{H^3}ds}\geq\underline pe^{-Ct}
$$
and
$$
  \sigma(x,y,z,t)\geq\underline\sigma e^{-\int_0^t(\|\text{div}_hv\|_\infty+\|\phi\|_\infty)ds}\geq\underline\sigma e^{-C\int_0^t(\|v\|_{H^3}+\|\phi\|_{H^2})ds}\geq \underline\sigma e^{-Ct},
$$
proving the conclusion.
\end{proof}

\begin{proposition}
\label{PROP-EGYINEQT}
Assume that all the assumptions stated in Proposition \ref{PROP-ESTPhis} hold. Then, it holds that
\begin{align*}
  &\frac{d}{dt}\|(v,\nabla\Delta v)\|_2^2+\mu\underline\sigma\|\Delta^2v\|_2^2
  \leq \nu C_*\|\partial_z\sigma\|_{H^2}\|v\|_{H^4}\|\partial_zv\|_{H^2}+C\left(1+\|v\|_{H^4}^\frac32\right),\\
  &\frac{d}{dt}\|(\sigma,\Delta\sigma)\|_2^2+\nu\underline\sigma\|\Delta\partial_z\sigma\|_2^2
  +\epsilon\|(\nabla_h\sigma,\nabla_h\Delta\sigma)\|_2^2
  \leq C\left(1+\|\partial_z\sigma\|_{H^2}^\frac32\right),\\
  &\frac{d}{dt}\|(p,\nabla_h\Delta_hp)\|_2^2+\epsilon\|(\nabla_hp,\nabla_h^2\Delta_hp)\|_2^2\leq C\left(1+\|v\|_{H^4}\right),
\end{align*}
for any $t\in(0,\mathcal T)$,
where $C_*$ is an absolute positive constant, and $C$ is a positive constant depending only on $\gamma$, $\nu$, $\mu$,
$\lambda$, $\underline p$, and $M$.
\end{proposition}

\begin{proof}
Noticing that $\int_0^1\tilde fdz=0$ and since $p$ is independent of $z$, it is clear from the expression of $\phi(v,p)$ in (\ref{phi}) that
$\int_0^1\phi(v,p)dz=0.$
Thus, recalling (\ref{EQw}) and by the boundary condition (\ref{BC2}), one has
\begin{equation}
  \label{BCw}
  w|_{z=0,1}=0.
\end{equation}
Besides, recalling (\ref{Q}) and using the boundary condition $\partial_zv|_{z=0,1}=0$, it holds that
\begin{equation}
  \label{BCphi2}
  \partial_z\phi(v,p)|_{z=0,1}=0.
\end{equation}
Taking $\partial_z$ to \eqref{EQv} and \eqref{EQs}, and recalling (\ref{BC2}), \eqref{BCw}, and \eqref{BCphi2}, one gets
\begin{equation}
  \label{BC3}
  \partial_z^3v|_{z=0,1}=0,\quad\partial_z^3\sigma|_{z=0,1}=0.
\end{equation}

We first carry out the energy inequalities for $v$. Multiplying \eqref{EQv'} with $v$, integrating the resultant over $\mathcal O$, one obtains by Proposition \ref{PROP-ESTPhis} that
\begin{eqnarray}
  \frac12\frac{d}{dt}\|v\|_2^2&=&\int\sigma(\mu\Delta v+(\mu+\lambda)\nabla_h\text{div}_hv)\cdot vdxdydz+\int\Phi_1\cdot vdxdydz \nonumber\\
  &\leq& C\|\sigma\|_2\|\nabla^2v\|_2\|v\|_\infty+\|\Phi_1\|_1\|v\|_\infty\nonumber\\
  &\leq& C\|\sigma\|_{2}\|v\|_{H^2}^2+\|v\|_{H^2}\leq C.\label{EQ16}
\end{eqnarray}
Applying $-\Delta$ to \eqref{EQv'}, multiplying the resultant with $\Delta^2v$, integrating over $\mathcal O$, and using the boundary
conditions \eqref{BC1}, \eqref{BC2}, and \eqref{BC3}, it follows from integrating by parts that
\begin{eqnarray}
  \frac12\frac{d}{dt}\|\nabla\Delta v\|_2^2
  &=&-\int\Delta(\mu\sigma\Delta v+(\mu+\lambda)\sigma\nabla_h\text{div}_hv+\Phi_1)\cdot\Delta^2vdxdydz\nonumber\\
  &=&-\int(\mu\sigma\Delta^2v+(\mu+\lambda)\sigma\nabla_h\text{div}_h\Delta v+\Delta\Phi_1)\cdot\Delta^2vdxdydz\nonumber\\
  &&-\int(\mu[\Delta,\sigma]\Delta v+(\mu+\lambda)[\Delta,\sigma]\nabla_h\text{div}_hv)\cdot\Delta^2v dxdydz. \label{EQ17}
\end{eqnarray}
Using the boundary conditions (\ref{BC1}), \eqref{BC2}, and \eqref{BC3}, it follows from integrating by parts that
\begin{eqnarray*}
  -\int\sigma\nabla_h\text{div}_h\Delta v\cdot\Delta^2vdxdydz
  =\int(\sigma\text{div}_h\Delta v\Delta^2\text{div}_hv+\text{div}_h\Delta v\nabla_h\sigma\cdot\Delta^2v)dxdydz\\
  =-\int[\sigma|\nabla\text{div}_h\Delta v|^2+\text{div}_h\Delta v(\nabla\sigma\cdot\nabla\Delta\text{div}_hv+\nabla_h\sigma\cdot\Delta^2v)]dxdydz.
\end{eqnarray*}
Substituting this into (\ref{EQ17}), recalling \eqref{ASSUM-LBsp}, applying Proposition \ref{PROP-ESTPhis}, and noticing that $\mu+\lambda>0$,
one deduces
\begin{eqnarray}
  &&\frac12\frac{d}{dt}\|\nabla\Delta v\|_2^2+0.5\underline\sigma \mu\|\Delta^2v\|_2^2\nonumber \\
  &\leq& C\int(|[\Delta,\sigma]\nabla^2v||\nabla^4v|+|\nabla\sigma||\nabla^3v||\nabla^4v|)dxdydz+ \|\Phi_1\|_{H^2}\|v\|_{H^4}\nonumber\\
  &\leq& C\int(|[\Delta,\sigma]\nabla^2v||\nabla^4v|+|\nabla\sigma||\nabla^3v||\nabla^4v|)dxdydz\nonumber\\
  &&+ (\nu C_*\|\partial_z\sigma\|_{H^2}\|\partial_zv\|_{H^2}+C)\|v\|_{H^4},\label{EQ18}
\end{eqnarray}
where $C_*$ is an absolute positive constant.
Applying \eqref{COME} and by assumption \eqref{ASSUM-APRM}, it follows from the H\"older, Sobolev, and Gagliardo-Nirenberg inequalities that
\begin{eqnarray}
  &&\int|[\Delta,\sigma]\nabla^2v||\nabla^4v|dxdydz\nonumber\\
  &\leq& \|[\Delta,\sigma]\nabla^2v\|_2\|v\|_{H^4} \leq C\left(\|\nabla\sigma\|_6\|\nabla^2v\|_{W^{1,3}}+\|\sigma\|_{H^2}\|\nabla^2v\|_\infty\right)\|v\|_{H^4}\nonumber\\
  &\leq& C\|\sigma\|_{H^2}\|v\|_{H^3}^\frac12\|v\|_{H^4}^\frac12\|v\|_{H^4}\leq C\|v\|_{H^4}^\frac32 \label{EQ19}
\end{eqnarray}
and
\begin{eqnarray}
  \int|\nabla\sigma||\nabla^3v||\nabla^4v|dxdydz
  &\leq&\|\nabla\sigma\|_6\|\nabla^3v\|_3\|\nabla^4v\|_2\nonumber\\
  &\leq& C\|\sigma\|_{H^2}\|v\|_{H^3}^\frac12\|v\|_{H^4}^\frac32 \leq C\|v\|_{H^4}^\frac32. \label{EQ20}
\end{eqnarray}
Substituting (\ref{EQ19}) and (\ref{EQ20}) into (\ref{EQ18}) and adding the resultant with (\ref{EQ16}), it follows
\begin{equation}
  \frac{d}{dt}\|(v,\nabla\Delta v)\|_2^2+\underline\sigma \mu\|\Delta^2v\|_2^2
  \leq C(1+\|v\|_{H^4}^\frac32)+\nu C_*\|\partial_z\sigma\|_{H^2}\|v\|_{H^4}\|\partial_zv\|_{H^2},\label{EQ34}
\end{equation}
proving the first conclusion.

Next, we carry out the energy inequality for $\sigma$. Multiplying \eqref{EQs'} with $\sigma$
and integrating over $\mathcal O$, it follows from integrating
by parts, the H\"older and Sobolev inequalities, (\ref{ASSUM-APRM}), and Proposition \ref{PROP-ESTPhis} that
\begin{eqnarray}
  \frac12\frac{d}{dt}\|\sigma\|_2^2+\epsilon\|\nabla_h\sigma\|_2^2
  &=&\int(\nu\sigma\partial_z^2\sigma-v\cdot\nabla_h\sigma+\Phi_2)\sigma dxdydz\nonumber \\
  &\leq& C\|\sigma\|_\infty(\|\sigma\|_2\|\partial_z^2\sigma\|_2+\|v\|_2\|\nabla_h\sigma\|_2+\|\Phi_2\|_1)\nonumber\\
  &\leq& C(\|\sigma\|_{H^2}^3+\|v\|_2\|\sigma\|_{H^2}^2+\|\sigma\|_{H^2}\|\Phi_2\|_1)\leq C. \label{EQ23}
\end{eqnarray}
Applying $\Delta$ to \eqref{EQs'}, multiplying the resultant with $\Delta\sigma$, using the boundary conditions \eqref{BC1}, \eqref{BC2}, and \eqref{BC3},
it follows from integrating by parts that
\begin{eqnarray*}
  &&\frac12\frac{d}{dt}\|\Delta\sigma\|_2^2+\epsilon\|\nabla_h\Delta\sigma\|_2^2 \nonumber\\
  &=&\nu\int(\sigma\Delta\partial_z^2\sigma+[\Delta,\sigma]\partial_z^2\sigma)\Delta\sigma dxdydz+\int\Delta(\Phi_2-v\cdot\nabla_h\sigma)\Delta\sigma dxdydz\nonumber\\
  &=&-\nu\int(\sigma|\Delta\partial_z\sigma|^2+\partial_z\sigma\Delta\partial_z\sigma\Delta\sigma)dxdydz\nonumber\\
  &&-\nu\int([\Delta,\partial_z\sigma]\partial_z\sigma\Delta\sigma+
  [\Delta,\sigma]\partial_z\sigma\Delta\partial_z\sigma) dxdydz\nonumber\\
  &&+\int\Delta(\Phi_2-v\cdot\nabla_h\sigma)\Delta\sigma dxdydz
\end{eqnarray*}
and thus
\begin{eqnarray}
  &&\frac12\frac{d}{dt}\|\Delta\sigma\|_2^2+\epsilon\|\nabla_h\Delta\sigma\|_2^2+\nu\int\sigma|\Delta\partial_z\sigma|^2dxdydz \nonumber\\
  &=&-\nu\int \partial_z\sigma\Delta\partial_z\sigma\Delta\sigma dxdydz
 -\nu\int([\Delta,\partial_z\sigma]\partial_z\sigma\Delta\sigma+
  [\Delta,\sigma]\partial_z\sigma\Delta\partial_z\sigma) dxdydz\nonumber\\
  &&+\int\Delta(\Phi_2-v\cdot\nabla_h\sigma)\Delta\sigma dxdydz.\label{EQ25}
\end{eqnarray}
Using \eqref{ASSUM-APRM}, it follows from the H\"older, Sobolev, and Gagliardo-Nirenberg inequalities that
\begin{eqnarray}
\left|\int\partial_z\sigma\Delta\partial_z\sigma\Delta\sigma dxdydz\right|\leq \|\partial_z\sigma\|_\infty\|\Delta\partial_z\sigma\|_2\|\Delta
\sigma\|_2\nonumber\\
\leq C\|\partial_z\sigma\|_{H^1}^\frac12\|\partial_z\sigma\|_{H^2}^\frac32\|\sigma\|_{H^2}\leq C\|\partial_z\sigma\|_{H^2}^\frac32.\label{EQ26}
\end{eqnarray}
Applying \eqref{COME} and recalling \eqref{ASSUM-APRM}, it follows from the H\"older, Sobolev, and Gagliardo-Nirenberg inequalities that
\begin{eqnarray}
  \left|\int[\Delta,\partial_z\sigma]\partial_z\sigma\Delta\sigma dxdydz\right|
  &\leq&\|[\Delta,\partial_z\sigma]\partial_z\sigma\|_2\|\Delta\sigma\|_2\nonumber\\
  &\leq& C(\|\nabla\partial_z\sigma\|_3\|\partial_z\sigma\|_{W^{1,6}}
  +\|\partial_z\sigma\|_{H^2}\|\partial_z\sigma\|_\infty)\|\sigma\|_{H^2}\nonumber\\
  &\leq& C\|\partial_z\sigma\|_{H^1}^\frac12\|\partial_z\sigma\|_{H^2}^\frac32\|\sigma\|_{H^2}\leq C\|\partial_z\sigma\|_{H^2}^\frac32
  \label{EQ27}\\
  \left|\int[\Delta,\sigma]\partial_z\sigma\Delta\partial_z\sigma dxdydz\right|&\leq&\|[\Delta,\sigma]\partial_z\sigma\|_2\|\Delta\partial_z
  \sigma\|_2\nonumber\\
  &\leq&C(\|\nabla\sigma\|_6\|\partial_z\sigma\|_{W^{1,3}}+\|\sigma\|_{H^2}\|\partial_z\sigma\|_\infty)\|\partial_z\sigma\|_{H^2}\nonumber\\
  &\leq&C\|\sigma\|_{H^2}\|\partial_z\sigma\|_{H^1}^\frac12\|\partial_z\sigma\|_{H^2}^\frac32\leq C\|\partial_z\sigma\|_{H^2}^\frac32,
  \label{EQ28}
\end{eqnarray}
and
\begin{eqnarray}
  &&-\int\Delta(v\cdot\nabla_h\sigma)\Delta\sigma dxdydz\nonumber\\
  &=&-\int(v\cdot\nabla_h\Delta\sigma\Delta\sigma+[\Delta,v\cdot\nabla_h]\sigma\Delta\sigma)dxdydz\nonumber\\
  &\leq&\frac12\int\text{div}_hv|\Delta\sigma|^2dxdydz +\|[\Delta,v\cdot\nabla_h]\sigma\|_2\|\Delta\sigma\|_2\nonumber\\
  &\leq&C\|v\|_{H^3}\|\sigma\|_{H^2}^2+C(\|\nabla v\|_\infty\|\nabla_h\sigma\|_{H^1}+\|v\|_{W^{2,6}}\|\nabla_h\sigma\|_3)
  \|\Delta\sigma\|_2\nonumber\\
  &\leq&C\|v\|_{H^3}\|\sigma\|_{H^2}^2\leq C.\label{EQ29}
\end{eqnarray}
By Proposition \ref{PROP-ESTPhis} and recalling \eqref{ASSUM-APRM}, it follows
\begin{equation}
  \int\Delta\Phi_2\Delta\sigma dxdydz\leq C\|\Phi_2\|_{H^2}\|\sigma\|_{H^2}\leq C\left(1+\|\partial_z\sigma\|_{H^2}^\frac32\right). \label{EQ30}
\end{equation}
Substituting (\ref{EQ26})--(\ref{EQ30}) into \eqref{EQ25} and adding the resultant with \eqref{EQ23}, it follows from
\eqref{ASSUM-LBsp} and the Young inequality that
\begin{equation}
  \frac{d}{dt}\|(\sigma,\Delta\sigma)\|_2^2+\epsilon\|(\nabla_h\sigma,\nabla_h\Delta\sigma)\|_2^2
  +\nu\underline\sigma\|\Delta\partial_z\sigma\|_2^2\leq C\left(1+\|\partial_z\sigma\|_{H^2}^\frac32\right), \label{EQ31}
\end{equation}
proving the second conclusion.

Finally, we carry out the energy inequality for $p$. Multiplying \eqref{EQp} with $p$ and integrating over $\mathbb T^2$, it following from
Proposition \ref{PROP-ESTPhis} and \eqref{ASSUM-APRM} that
\begin{eqnarray}
  &&\frac12\frac{d}{dt}\|p\|_2^2+\epsilon\|\nabla_hp\|_2^2= \int(\Phi_3-\bar v\cdot\nabla_hp)pdxdy\nonumber
  \\
  &\leq& \|\Phi_3\|_1\|p\|_\infty+\|\bar v\|_2\|\nabla_hp\|_2\|p\|_\infty\leq C(\|p\|_{H^2}+\|v\|_2\|p\|_{H^2}^2)\leq C.
  \label{EQ32}
\end{eqnarray}
Applying $\nabla_h\Delta_h$ to \eqref{EQp}, multiplying the resultant with $\nabla_h\Delta_hp$, and integrating over $\mathbb T^2$, one
obtains by integrating by parts, \eqref{COME}, Proposition \ref{PROP-ESTPhis}, and \eqref{ASSUM-APRM} that
\begin{eqnarray*}
  &&\frac12\frac{d}{dt}\|\nabla_h\Delta_hp\|_2^2+\epsilon\|\nabla_h^2\Delta_hp\|_2^2\nonumber\\
  &=&\int\Big(\nabla_h\Delta_h\Phi_3-(\bar v\cdot\nabla_h)\nabla_h\Delta_hp-[\nabla_h\Delta_h,\bar v\cdot\nabla_h]p\Big)\cdot\nabla_h\Delta_hpdxdy \nonumber\\
  &\leq&\|\Phi_3\|_{H^3}\|p\|_{H^3}+\frac12\int\text{div}_h\bar v|\nabla_h\Delta_hp|^2dxdy+\|[\nabla_h\Delta_h,\bar v\cdot\nabla_h]p\|_2\|\nabla_h\Delta_h
  p\|_2\nonumber\\
  &\leq&C(1+\|v\|_{H^4})\|p\|_{H^3}+C\|\nabla v\|_\infty\|p\|_{H^3}^2\nonumber\\
  &&+C(\|\nabla_h\bar v\|_\infty\|\nabla_hp\|_{H^2}+\|\bar v\|_{H^3}\|\nabla_hp\|_\infty)\|p\|_{H^3}\nonumber\\
  &\leq&C(1+\|v\|_{H^4})\|p\|_{H^3}+C\|v\|_{H^3}\|p\|_{H^3}^2\leq C(1+\|v\|_{H^4}).
\end{eqnarray*}
Summing this with \eqref{EQ32} leads to
\begin{equation}
  \frac{d}{dt}\|(p,\nabla_h\Delta_hp)\|_2^2+\epsilon\|(\nabla_hp,\nabla_h^2\Delta_hp)\|_2^2
  \leq C(1+\|v\|_{H^4}),\label{EQ33}
\end{equation}
proving the third conclusion.
\end{proof}

\begin{proposition}
\label{PROP-EST-UnderM}
Assume that all the assumptions stated in Proposition \ref{PROP-ESTPhis} hold and let $X_0$ be an arbitrary positive constant such that
$$
\|v_0\|_{H^3}^2+\|\sigma_0\|_{H^2}^2+\|p_0\|_{H^3}^2\leq X_0.
$$
Then, it holds that
\begin{align*}
(\|v\|_{H^3}^2&+\|\sigma\|_{H^2}^2+\|p\|_{H^3}^2)(t)+\epsilon\int_0^t(\|\nabla_h\sigma\|_{H^2}^2+\|\nabla_hp\|_{H^3}^2)ds\nonumber\\
  &+\int_0^t(\|v\|_{H^4}^2+\|\partial_z\sigma\|_{H^2}^2)ds \leq e^{C_\sharp(X_0+Ct) },
\end{align*}
for any $t\in(0,\mathcal T)$,
where constant $C_\sharp\geq1$ depends only on $\underline\sigma$, $\nu,$ and $\mu,$ while constant $C$ depends only on $\gamma$, $\nu$, $\mu$,
$\lambda$, $\underline\sigma$, $\underline p$, and $M$.
\end{proposition}

\begin{proof}
For simplicity, throughout the proof of this proposition, we use $C_a$ to denote an absolute positive constant, which in particular is
independent of $M$, while $C$ is a positive constant depending only on $\gamma$, $\nu$, $\mu$,
$\lambda$, $\underline\sigma$, $\underline p$, and $M$.

Recalling \eqref{BC1}, \eqref{BC2}, and \eqref{BC3}, which imply $\partial_z\Delta v|_{z=0,1}=0$, and using
\eqref{ASSUM-APRM}, it follows from the elliptic estimate that
\begin{eqnarray}
  \|v\|_{H^4}&\leq&{C_a} (\|\Delta v\|_{H^2}+\|v\|_2)\leq{C_a} (\|\Delta^2v\|_2+\|\Delta v\|_2+\|v\|_2)\nonumber\\
  &=& {C_a} (\|\Delta^2v\|_2+\|(v,\Delta v)\|_2)\label{H4vELLIP}
\end{eqnarray}
and
\begin{equation}
  \label{H2zvsELLIP}
  \|\partial_zv\|_{H^2}\leq {C_a}\|\Delta\partial_zv\|_2,\quad
\|\partial_z\sigma\|_{H^2}\leq {C_a}\|\Delta\partial_z\sigma\|_2.
\end{equation}
Note that the
lower order norms are not required in the estimates of \eqref{H2zvsELLIP} as the solution to the Laplace equation subject to
homogeneous Dirichlet boundary conditions on top and bottom and horizontally periodic condition is unique.
Thanks to \eqref{H4vELLIP} and using \eqref{ASSUM-APRM}, one gets
\begin{equation}
  \|v\|_{H^4}\leq {C_a}(\|\Delta^2v\|_2+C).\label{H4vELLIP-UnderM}
\end{equation}

With the aid of \eqref{H2zvsELLIP}, one gets from Proposition \ref{PROP-EGYINEQT} that
\begin{eqnarray*}
  \frac{d}{dt}\|(\sigma,\Delta\sigma)\|_2^2+\epsilon\|(\nabla_h\sigma,\nabla_h\Delta\sigma)\|_2^2+\nu\underline\sigma
  \|\Delta\partial_z\sigma\|_2^2\nonumber\\
  \leq C\left(1+\|\Delta\partial_z\sigma\|_2^\frac32\right)\leq\frac{\nu\underline\sigma}{2}\|\Delta\partial_z\sigma\|_2^2+C
\end{eqnarray*}
and thus
\begin{align}
  \|(\sigma,\Delta\sigma)\|_2^2(t)&+\epsilon\int_0^t\|(\nabla_h\sigma,\nabla_h\Delta\sigma)\|_2^2ds\nonumber\\
  &+\frac{\nu\underline\sigma}{2}\int_0^t\|
  \Delta\partial_z\sigma\|_2^2ds\leq\|(\sigma_0,\Delta\sigma_0)\|_2^2+Ct. \label{ESTs-UnderM}
\end{align}

By Proposition \ref{PROP-EGYINEQT}, it follows from \eqref{H2zvsELLIP}, \eqref{H4vELLIP-UnderM}, and the Young inequality that
\begin{eqnarray*}
  &&\frac{d}{dt}\|(v,\nabla\Delta v)\|_2^2+\mu\underline\sigma\|\Delta^2v\|_2^2\\
  &\leq& C\left(1+\|\Delta^2v\|_2^\frac32\right)+\nu{C_a}\|\partial_z\Delta\sigma\|_2(\|\Delta^2v\|_2+C)\|\partial_z\Delta v\|_2\\
  &\leq& \frac{\mu\underline\sigma}{2}\|\Delta^2v\|_2^2+\frac{\nu^2{C_a}}{\mu\underline\sigma}\|\Delta\partial_z\sigma\|_2^2
  \|\partial_z\Delta v\|_2^2+C
\end{eqnarray*}
and thus
\begin{eqnarray*}
  \frac{d}{dt}\|(v,\nabla\Delta v)\|_2^2+\frac{\underline\sigma \mu}{2}\|\Delta^2v\|_2^2
   \leq  \frac{\nu^2{C_a}}{\mu\underline\sigma}\|\Delta\partial_z\sigma\|_2^2\|\nabla\Delta v\|_2^2+C.
\end{eqnarray*}
Combining this with \eqref{ESTs-UnderM} leads to
\begin{eqnarray}
  &&\|(v,\nabla\Delta v)\|_2^2(t)+\frac{\underline\sigma \mu}{2}\int_0^t\|\Delta^2v\|_2^2ds\nonumber\\
  &\leq& (\|(v_0,\nabla\Delta v_0)\|_2^2+Ct)\exp\left\{\frac{\nu^2{C_a}}{\mu\underline\sigma}\int_0^t\|\Delta\partial_z\sigma\|_2^2 ds \right\}\nonumber\\
  &\leq& e^{\frac{2\nu{C_a}}{\mu\underline\sigma^2}(\|(\sigma_0,\Delta\sigma_0)\|_2^2+Ct) }(\|(v_0,\nabla\Delta v_0)\|_2^2+Ct).
  \label{ESTv-UnderM}
\end{eqnarray}

It follows from Proposition \ref{PROP-EGYINEQT}, \eqref{H4vELLIP-UnderM}, and the Young inequality that
\begin{eqnarray*}
  &&\frac{d}{dt}\|(p,\nabla_h\Delta_hp)\|_2^2+\epsilon\|(\nabla_hp,\nabla_h^2\Delta_hp)\|_2^2\\
  &\leq&C(1+\|v\|_{H^4})\leq C(1+\|\Delta^2v\|_2)
  \leq \frac{ \mu\underline\sigma}{2}\|\Delta^2v\|_2^2+C,
\end{eqnarray*}
from which by \eqref{ESTv-UnderM} one gets
\begin{eqnarray}
&&\|(p,\nabla_h\Delta_hp)\|_2^2(t)+\epsilon\int_0^t\|(\nabla_hp,\nabla_h^2\Delta_hp)\|_2^2ds\nonumber\\
&\leq&\|(p_0,\nabla_h\Delta_hp_0)\|_2^2+ e^{\frac{2\nu {C_a}}{\mu\underline\sigma^2}(\|(\sigma_0,\Delta\sigma_0)\|_2^2+Ct) }(\|(v_0,\nabla\Delta v_0)\|_2^2+Ct)\nonumber\\
&\leq& e^{\frac{2\nu {C_a}}{\mu\underline\sigma^2}(\|(\sigma_0,\Delta\sigma_0)\|_2^2+Ct) }(\|(p_0,\nabla_h\Delta_hp_0)\|_2^2+\|(v_0,\nabla\Delta v_0)\|_2^2+Ct).
\label{ESTp-UnderM}
\end{eqnarray}

Combining \eqref{ESTs-UnderM} with \eqref{ESTv-UnderM}, together with \eqref{ESTp-UnderM}, one obtains
\begin{eqnarray*}
  &&\|(v,\nabla\Delta v, \sigma, \Delta\sigma)\|_2^2(t)+\|(p,\nabla_h\Delta_hp)\|_2^2(t)\nonumber\\
  &&+\epsilon\int_0^t(\|(\nabla_h\sigma,\nabla_h\Delta\sigma)\|_2^2+\|(\nabla_hp,\nabla_h^2\Delta_hp)\|_2^2ds\nonumber\\
  &&+\frac{\underline\sigma\mu}{2}\int_0^t\|\Delta^2v\|_2^2ds +\frac{\nu\underline\sigma}{2}\int_0^t\|\Delta\partial_z\sigma\|_2^2ds\nonumber \\
  &\leq& e^{\frac{2\nu {C_a}}{\mu\underline\sigma^2}(\|(\sigma_0,\Delta\sigma_0)\|_2^2+Ct) }(\|(v_0,\nabla\Delta v_0,\sigma_0,\Delta\sigma_0)\|_2^2+\|(p_0,\nabla_h\Delta_hp_0)\|_2^2+Ct),
\end{eqnarray*}
from which, by the elliptic estimate, the conclusion follows.
\end{proof}

\begin{proposition}
  \label{PROP-EST-UnderM-Time}
  Assume that all the assumptions stated in Proposition \ref{PROP-ESTPhis} hold and that $\epsilon\in(0,1)$. Then, it holds that
  \begin{equation*}
    \int_0^t(\|\partial_tv\|_{H^2}^2+\|\partial_t\sigma\|_{H^1}^2+\|\partial_tp\|_{H^2}^2)ds\leq e^{C(1+t)},
  \end{equation*}
  for any $t\in(0,\mathcal T)$, where $C$ is a positive constant depending only on $\gamma$, $\nu$, $\mu,$
$\lambda$, $\underline\sigma$, $\underline p$, $M$, and $X_0$, with $X_0$ being the same as in Proposition \ref{PROP-EST-UnderM}.
\end{proposition}

\begin{proof}
  Applying \eqref{ALGHk} and Proposition \ref{PROP-ESTPhis}, it follows from \eqref{EQv'} and \eqref{ASSUM-APRM} that
  \begin{eqnarray}
    \|\partial_tv\|_{H^2}&\leq&\|\Phi_1\|_{H^2}+C\|\sigma\nabla^2v\|_{H^2}\nonumber\\
    &\leq&C(\|\partial_zv\|_{H^2}\|\partial_z\sigma\|_{H^2}+\|\sigma\|_{H^2}\|\nabla^2 v\|_{H^2}+1)\nonumber\\
    &\leq&C(\|\partial_z\sigma\|_{H^2}+\|v\|_{H^4}+1).\label{910-1}
  \end{eqnarray}
  Recalling the expression of \eqref{Phi2} and applying \eqref{ALGmq}, it follows from the Sobolev inequality, \eqref{H2phi}, \eqref{H1w}, \eqref{H2w}, and \eqref{ASSUM-APRM} that
  \begin{eqnarray*}
    &&\|\Phi_2-v\cdot\nabla_h\sigma\|_{H^1}\leq\|\Phi_2\|_{H^1}+\|v\cdot\nabla_h\sigma\|_{H^1}\\
    &\leq&\|w\partial_z\sigma\|_{H^1}+\|\sigma(\text{div}_hv-\phi(v,p))\|_{H^1}+\|v\cdot\nabla_h\sigma\|_{H^1}\\
    &\leq&C(\|w\|_\infty\|\partial_z\sigma\|_{H^1}+\|w\|_{H^1}\|\partial_z\sigma\|_\infty)+C[\|\sigma\|_\infty(\|\text{div}_hv\|_{H^1}+\|\phi(v,p)\|_{H^1})\\
    &&+\|\sigma\|_{H^1}(\|\text{div}_hv\|_\infty+\|\phi(v,p)\|_\infty)]+C(\|v\|_\infty\|\nabla_h\sigma\|_{H^1}+\|v\|_{W^{1,3}}\|\nabla_h\sigma\|_6)\\
    &\leq&C(\|w\|_{H^2}\|\sigma\|_{H^2}+\|w\|_{H^1}\|\partial_z\sigma\|_{H^2})+C\|\sigma\|_{H^2}(\|v\|_{H^3}+\|\phi(v,p)\|_{H^2}+\|v\|_{H^2})\\
    &\leq&C(\|\partial_z\sigma\|_{H^2}+1).
  \end{eqnarray*}
  Thanks to this and using \eqref{ALGmq}, it follows from \eqref{EQs'} and the Sobolev inequality that
  \begin{eqnarray}
    \|\partial_t\sigma\|_{H^1}&\leq&\|\Phi_2-v\cdot\nabla\sigma\|_{H^1}+\epsilon\|\Delta_h\sigma\|_{H^1}+\nu\|\sigma\partial_z^2\sigma\|_{H^1}\nonumber\\
    &\leq&C(\|\partial_z\sigma\|_{H^2}+1)+\epsilon\|\nabla_h\sigma\|_{H^2}+C(\|\sigma\|_\infty\|\partial_z^2\sigma\|_{H^1}+\|\sigma\|_{W^{1,3}}\|\partial_z^2\sigma\|_6)\nonumber\\
    &\leq&C(\|\partial_z\sigma\|_{H^2}+1)+\epsilon\|\nabla_h\sigma\|_{H^2}+C\|\sigma\|_{H^2}\|\partial_z^2\sigma\|_{H^1}\nonumber\\
    &\leq&C(\|\partial_z\sigma\|_{H^2}+1)+\epsilon\|\nabla_h\sigma\|_{H^2}.\label{910-2}
  \end{eqnarray}
  Applying \eqref{ALGHk} and Proposition \ref{PROP-ESTPhis}, it follows from \eqref{ASSUM-APRM} that
  \begin{eqnarray*}
    \|\Phi_3-\bar v\cdot\nabla_hp\|_{H^2}&\leq&\|\Phi_3\|_{H^2}+\|\bar v\cdot\nabla_hp\|_{H^2}\\
    &\leq&C(1+\|v\|_{H^4})+C\|\bar v\|_{H^2}\|\nabla_hp\|_{H^2}\\
    &\leq&C(1+\|v\|_{H^4}),
  \end{eqnarray*}
  from which, by \eqref{EQp'}, one obtains
  \begin{eqnarray}
    \|\partial_tp\|_{H^2}&\leq&\|\Phi_3-\bar v\cdot\nabla_hp\|_{H^2}+\epsilon\|\Delta_hp\|_{H^2}\nonumber\\
    &\leq&C(1+\|v\|_{H^4})+\epsilon\|\nabla_hp\|_{H^3}. \label{910-3}
  \end{eqnarray}
  Thanks to \eqref{910-1}, \eqref{910-2}, and \eqref{910-3}, one obtains by Proposition \ref{PROP-EST-UnderM} that
  \begin{eqnarray*}
    &&\int_0^t(\|\partial_tv\|_{H^2}^2+\|\partial_t\sigma\|_{H^1}^2+\|\partial_tp\|_{H^2}^2)ds\\
    &\leq&C\int_0^t(\|v\|_{H^4}^2+\|\partial_z\sigma\|_{H^2}^2)ds+Ct+2\epsilon^2\int_0^t(\|\nabla_h\sigma\|_{H^2}^2+\|\nabla_hp\|_{H^3}^2)ds\\
    &\leq& Ce^{C_\sharp(X_0+Ct)}+Ct\leq e^{C(1+t)},
  \end{eqnarray*}
  proving the conclusion.
\end{proof}

\subsection{$\epsilon$-dependent conditional energy estimates}
\begin{proposition}
\label{PROP-EST-Depep}
Assume that all the assumptions stated in Proposition \ref{PROP-ESTPhis} hold. Then, it holds that
\begin{align}
\|\sigma\|_{H^3}(t)+\int_0^t\|\sigma\|_{H^4}^2ds\leq e^{e^{C_\epsilon(1+t)}},
\end{align}
for any $t\in(0,\mathcal T)$,
where $C_\epsilon$ is a positive constant depending only on $\epsilon,$ $\gamma$, $\nu$, $\mu$,
$\lambda$, $\underline\sigma$, $\underline p$, $\|(v_0,\sigma_0)\|_{H^3}^2+\|p_0\|_{H^3}^2$, and $M$.
\end{proposition}

\begin{proof}
Applying $-\Delta$ to \eqref{EQs'}, multiplying the resultant with $\Delta^2\sigma$, using the boundary conditions \eqref{BC1}, \eqref{BC2}, and \eqref{BC3}, it follows from integrating by parts that
\begin{eqnarray}
  &&\frac12\frac{d}{dt}\|\nabla\Delta\sigma\|_2^2+\epsilon\|\nabla_h\nabla\Delta\sigma\|_2^2\nonumber\\
  &=&\int\Delta(v\cdot\nabla_h\sigma-\Phi_2)\Delta^2\sigma dxdydz-\nu\int\Delta(\sigma\partial_z^2\sigma)\Delta^2\sigma dxdydz. \label{920-1}
\end{eqnarray}
Applying \eqref{ALGHk} and recalling \eqref{ASSUM-APRM}, it follows from Proposition \eqref{PROP-ESTPhis} that
\begin{eqnarray*}
  \|\Phi_2-v\cdot\nabla_h\sigma\|_{H^2}&\leq&\|\Phi_2\|_{H^2}+\|v\cdot\nabla_h\sigma\|_{H^2}\nonumber\\
  &\leq& C\left(1+\|\partial_z\sigma\|_{H^2}^\frac32\right)+C\|v\|_{H^2}\|\nabla_h\sigma\|_{H^2}\nonumber\\
  &\leq& C\left(1+\|\partial_z\sigma\|_{H^2}^\frac32+\|\nabla_h\sigma\|_{H^2}\right)
\end{eqnarray*}
and thus
\begin{eqnarray}
  \int\Delta(v\cdot\nabla_h\sigma-\Phi_2)\Delta^2\sigma dxdydz\leq C\|\Phi_2-v\cdot\nabla_h\sigma\|_{H^2}\|\Delta^2\sigma\|_2\nonumber\\
  \leq C\left(1+\|\partial_z\sigma\|_{H^2}^\frac32+\|\nabla_h\sigma\|_{H^2}\right)\|\Delta^2\sigma\|_2. \label{EQ35}
\end{eqnarray}
Using the boundary conditions \eqref{BC1}, \eqref{BC2}, and \eqref{BC3}, it follows from integrating by parts and \eqref{ASSUM-LBsp}
that
\begin{eqnarray}
  &&-\int\Delta(\sigma\partial_z^2\sigma)\Delta^2\sigma dxdydz\nonumber\\
  &=&-\int(\sigma\Delta\partial_z^2\sigma+[\Delta,\sigma]\partial_z^2\sigma)\Delta^2\sigma dxdydz\nonumber\\
  &=&\int(\sigma\nabla\Delta\partial_z^2\sigma\cdot\nabla\Delta\sigma+\Delta\partial_z^2\sigma\nabla\sigma\cdot\nabla\Delta\sigma-[\Delta,\sigma]\partial_z^2\sigma\Delta^2\sigma)dxdydz\nonumber\\
  &=&-\int(\sigma\nabla\Delta\partial_z\sigma\cdot\nabla\Delta\partial_z\sigma+\partial_z\sigma\nabla\Delta\partial_z
  \sigma\cdot\nabla\Delta\sigma)dxdydz\nonumber\\
  &&+\int(\Delta\partial_z^2\sigma\nabla\sigma\cdot\nabla\Delta\sigma-[\Delta,\sigma]\partial_z^2\sigma\Delta^2\sigma)dxdydz\nonumber\\
  &\leq&-0.5\underline\sigma\int|\nabla\Delta\partial_z\sigma|^2dxdydz+C\int|\nabla\sigma||\nabla\Delta\partial_z\sigma|\nabla\Delta\sigma|dxdydz\nonumber\\
  &&+C\int|[\Delta,\sigma]\partial_z^2\sigma||\Delta^2\sigma|dxdydz.\label{EQ36}
\end{eqnarray}
Substituting \eqref{EQ35} and \eqref{EQ36} into \eqref{920-1} yields
\begin{eqnarray}
  &&\frac12\frac{d}{dt}\|\nabla\Delta\sigma\|_2^2+\epsilon\|\nabla_h\nabla\Delta v\|_2^2+\frac{\nu\underline\sigma}{2}\|\nabla\Delta\partial_z\sigma\|_2^2\nonumber\\
  &\leq& C\int(|\nabla\sigma||\nabla\Delta\sigma||\nabla\Delta\partial_z\sigma|+|[\Delta,\sigma]\partial_z^2\sigma||\Delta^2\sigma|)dxdydz\nonumber\\
  &&+C\left(1+\|\partial_z\sigma\|_{H^2}^\frac32+\|\nabla_h\sigma\|_{H^2}\right)\|\Delta^2\sigma\|_2,
\end{eqnarray}
from which applying \eqref{COME} and using the H\"older, Sobolev, and Young inequalities, one obtains that
\begin{eqnarray*}
  &&\frac12\frac{d}{dt}\|\nabla\Delta\sigma\|_2^2+\epsilon\|\nabla_h\nabla\Delta v\|_2^2+\frac{\nu\underline\sigma}{2}\|\nabla\Delta\partial_z\sigma\|_2^2\nonumber\\
  &\leq&C\|\nabla\sigma\|_\infty\|\nabla\Delta\sigma\|_2\|\nabla\Delta\partial_z\sigma\|_2+C(\|\nabla\sigma\|_\infty\|\partial_z^2\sigma\|_{H^1}+\|\sigma\|_{W^{2,3}}\|\partial_z^2\sigma\|_6)\|\Delta^2\sigma\|_2\nonumber\\
  &&+C(1+\|\nabla\sigma\|_{H^2}^2)\|\Delta^2\sigma\|_2\nonumber\\
  &\leq&C\|\sigma\|_{H^2}^\frac12\|\sigma\|_{H^3}^\frac32\|\nabla\Delta\partial_z\sigma\|_2
  +C(1+\|\sigma\|_{H^3}^2)\|\Delta^2\sigma\|_2\nonumber\\
  &\leq&\frac12\min\left\{\epsilon,\frac{\nu\underline\sigma}{2}\right\}\|\nabla^2\Delta v\|_2^2+C_\epsilon\left(\|\sigma\|_{H^3}^4+1\right).
\end{eqnarray*}
Thus,
\begin{equation}
  \frac{d}{dt}\|\nabla\Delta\sigma\|_2^2+\epsilon\|\nabla_h\nabla\Delta v\|_2^2+\frac{\nu\underline\sigma}{2}\|\nabla\Delta\partial_z\sigma\|_2^2
  \leq C_\epsilon(\|\sigma\|_{H^3}^4+1). \label{EQ37}
\end{equation}
Recalling the boundary conditions \eqref{BC1} and \eqref{BC2}, it follows from the elliptic estimate, the Poincar\'e inequality, and \eqref{ASSUM-APRM} that
$$
\|\sigma\|_{H^3}\leq C(\|\Delta\sigma\|_{H^1}+\|\sigma\|_2)\leq C(\|\nabla\Delta\sigma\|_2+1).
$$
Substituting this into \eqref{EQ37} yields
$$
  \frac{d}{dt}\|\nabla\Delta\sigma\|_2^2+\epsilon\|\nabla_h\nabla\Delta v\|_2^2+\frac{\nu\underline\sigma}{2}\|\nabla\Delta\partial_z\sigma\|_2^2
  \leq C_\epsilon(\|\nabla\Delta\sigma\|_2^4+1)
$$
from which, by the Gr\"onwall inequality and Proposition \ref{PROP-EST-UnderM}, one deduces
\begin{eqnarray*}
  \|\nabla\Delta\sigma\|_2^2(t)+\epsilon\int_0^t\|\nabla_h\nabla\Delta v\|_2^2ds+\frac{\nu\underline\sigma}{2}\int_0^t\|\nabla\Delta\partial_z\sigma\|_2^2ds\\
  \leq e^{C_\epsilon\int_0^t\|\nabla\Delta\sigma\|_2^2ds}(\|\nabla\Delta\sigma_0\|_2^2+C_\epsilon t)\leq e^{e^{C_\epsilon(1+t)}},
\end{eqnarray*}
for some positive constant $C_\epsilon$ depending only on $\epsilon$, $\nu$, $\gamma$, $\mu$, $\lambda$, $\underline\sigma$, $\underline p$,
$\|(v_0,\sigma_0)\|_{H^3}^2+\|p_0\|_{H^3}^2$, and $M$.
With the aid of the above estimate, the conclusion follows from the elliptic estimate and \eqref{ASSUM-APRM}.
\end{proof}

\subsection{$\epsilon$-independent a priori estimates}

\begin{proposition}
  \label{PROP-APR}
Given $(v_{0}, \sigma_{0}, p_{0})$ satisfying (\ref{ASSUMIC1})--(\ref{ASSUMIC2}). Let $X_0$ be an arbitrary constant such that
$\|v_{0}\|_{H^3}^2+\|\sigma_{0}\|_{H^2}^2+\|p_0\|_{H^3}^2\leq X_0$. Then, there is a positive time $\mathcal T_0$,
depending only on $\gamma$, $\nu$, $\mu$, $\lambda$, $\underline\sigma$, $\underline p$, and
  $X_0$, such that system (\ref{EQv})--(\ref{EQp}),
  subject to (\ref{BC1})--(\ref{IC}), has a unique local solution $(v,\sigma,p)$ on $\mathcal O\times(0,\mathcal T_0)$,
  satisfying
  \begin{eqnarray*}
  &&\inf_{(x,y,z,t)\in\mathcal O\times[0,\mathcal T_0]}\sigma\geq0.5\underline\sigma,\quad \inf_{(x,y,t)\in\mathbb T^2\times[0,\mathcal T_0]}p\geq0.5\underline p,\\
  &&\sup_{0\leq t\leq\mathcal T_0}(\|v\|_{H^3}^2+\|\sigma\|_{H^2}^2+\|p\|_{H^3}^2)(t)+\epsilon\int_0^{\mathcal T_0}(\|\nabla_h\sigma\|_{H^2}^2+\|\nabla_hp\|_{H^3}^2)dt\nonumber\\
  &&\, \, ~~~+\int_0^{\mathcal T_0}(\|v\|_{H^4}^2+\|\partial_z\sigma\|_{H^2}^2+\|\partial_tv\|_{H^2}^2+\|\partial_t\sigma\|_{H^1}^2+\|\partial_tp\|_{H^2}^2)dt \leq K_0,
  \end{eqnarray*}
  for a positive constant $K_0$ depending only on $\gamma$, $\nu$, $\mu$, $\lambda$, $\underline\sigma$, $\underline p$, and
  $X_0$.
\end{proposition}

\begin{proof}
Let $(v,\sigma,p)$ be the unique local solution obtained in Proposition \ref{PROPLOC-e}. By iteratively applying Proposition \ref{PROPLOC-e},
one can extend it uniquely to the maximal time of existence $\mathcal T_\text{max}$ which is characterized as
\begin{equation}\label{TE*}
  \varlimsup_{t\rightarrow \mathcal T_\text{max}}\left(\|(v,\sigma)\|_{H^3(\mathcal O)}+\|p\|_{H^3(\mathbb T^2)}+\left\|\frac1\sigma\right\|_\infty+
  \left\|\frac1p\right\|_\infty\right)=\infty,
\end{equation}
in case that $\mathcal T_\text{max}<\infty$. Denote
$$
M_0=2e^{C_\sharp(X_0+1)},
$$
with $C_\sharp$ being the positive constant as in Proposition \ref{PROP-EST-UnderM}, which depends only on $\underline\sigma$, $\nu,$
and $\mu$. Set
\begin{align*}
  \mathcal T_\sharp:=&\max\bigg\{\mathcal T\in[0,\mathcal T_\text{max})\bigg|\sup_{0\leq t\leq\mathcal T}(\|v\|_{H^3}^2+\|\sigma\|_{H^2}^2+\|p\|_{H^3}^2)\leq M_0,\\
  &\inf_{(x,y,z,t)\in\mathcal O\times[0,T]}\sigma\geq0.5\underline\sigma,\quad \mbox{and}\quad\inf_{(x,y,t)\in\mathbb T^2\times[0, T]}p\geq0.5\underline p\bigg\}.
\end{align*}

Note that if $\mathcal T_\sharp=\infty$, then one can simply set $\mathcal T_0=1$ and the conclusion follows directly from Proposition \ref{PROP-EST-UnderM} and Proposition \ref{PROP-EST-UnderM-Time}.
Therefore, in the rest of the proof, we assume that $\mathcal T_\sharp<\infty$.

By Proposition \ref{PROP-EST-Depep}, it holds that
$$
\sup_{0\leq t\leq\mathcal T_\sharp}\|\sigma\|_{H^3}(t)+\int_0^{\mathcal T_\sharp}\|\sigma\|_{H^4}^2(t)dt
\leq e^{e^{C_\epsilon(1+\mathcal T_\sharp)}},
$$
for a positive constant $C_\epsilon$. Hence,
\begin{equation}\label{916}
  \sup_{0\leq t\leq\mathcal T_\sharp}\left(\|(v,\sigma)\|_{H^3}+\|p\|_{H^3}+\left\|\frac1\sigma\right\|_\infty+
  \left\|\frac1p\right\|_\infty\right)\leq M_0+\frac2{\underline\sigma}+\frac2{\underline p}+ e^{e^{C_\epsilon(1+\mathcal T_\sharp)}}.
\end{equation}

We claim that $\mathcal T_\sharp<\mathcal T_\text{max}$. To this end, since $\mathcal T_\sharp<\infty$, it suffices to consider the case that
$\mathcal T_\text{max}<\infty$. Assume by contradiction that $\mathcal T_\sharp\geq\mathcal T_\text{max}$.
It is clear from the definition of $\mathcal T_\sharp$ that $\mathcal T_\sharp\leq\mathcal T_\text{max}$.
Thus, $\mathcal T_\text{max}=\mathcal T_\sharp$. Then, \eqref{916} implies
\begin{equation*}
  \varlimsup_{t\rightarrow\mathcal T_\text{max}}\left(\|(v,\sigma)\|_{H^3}+\|p\|_{H^3}+\left\|\frac1\sigma\right\|_\infty+
  \left\|\frac1p\right\|_\infty\right)(t)<\infty,
\end{equation*}
contradicting to \eqref{TE*}. Therefore, the claim holds.

Recalling the definition of $\mathcal T_\sharp$, it follows from
Proposition \ref{PROP-LBsp-UnderM}, Proposition \ref{PROP-EST-UnderM}, and Proposition \ref{PROP-EST-UnderM-Time} that
\begin{eqnarray}
&&\inf_{(x,y)\in\mathbb T^2}p\geq\underline pe^{-C_0t},\quad \inf_{(x,y,z)\in\mathcal O}\sigma \geq\underline\sigma e^{-C_0t}, \label{LBsp}\\
&&\int_0^t(\|\partial_tv\|_{H^2}^2+\|\partial_t\sigma\|_{H^1}^2+\|\partial_tp\|_{H^2}^2)ds\leq e^{C_0(1+t)},\label{APRESTTIME}
\end{eqnarray}
and
\begin{align}
(\|v\|_{H^3}^2&+\|\sigma\|_{H^2}^2+\|p\|_{H^3}^2)(t)+\epsilon\int_0^t(\|\nabla_h\sigma\|_{H^2}^2+\|\nabla_hp\|_{H^3}^2)ds\nonumber\\
  &+\int_0^t(\|v\|_{H^4}^2+\|\partial_z\sigma\|_{H^2}^2)ds \leq e^{C_\sharp(X_0+C_0t) },\label{Hkvsp}
\end{align}
for any $t\in[0,\mathcal T_\sharp]$,
where constant $C_\sharp\geq1$ depends only on $\underline\sigma, \nu,$ and $\mu,$ while constant $C_0$ depends only on $\gamma$, $\nu$, $\mu$,
$\lambda$, $\underline\sigma$, $\underline p$, and $X_0$.

We claim that
\begin{equation}\label{T0}
\mathcal T_\sharp>\mathcal T_0:=\min\left\{\frac{1}{C_0},-\frac{\ln 0.6}{C_0}\right\}=-\frac{\ln 0.6}{C_0}.
\end{equation}
Assume by contradiction that $\mathcal T_\sharp\leq\mathcal T_0$. Then, for any $t\in[0,\mathcal T_\sharp]$, it follows from \eqref{LBsp} and \eqref{Hkvsp} that
\begin{eqnarray*}
  &&\inf_{(x,y)\in\mathbb T^2}p\geq \underline pe^{-C_0t}\geq \underline pe^{-C_0\mathcal T_\sharp}\geq \underline pe^{-C_0\mathcal T_0}= 0.6\underline p,\\
  &&\inf_{(x,y,z)\in\mathcal O}\sigma\geq \underline \sigma e^{-C_0t}\geq \underline \sigma
  e^{-C_0\mathcal T_\sharp}\geq \underline \sigma e^{-C_0\mathcal T_0}= 0.6\underline\sigma,
\end{eqnarray*}
and
$$
  \|v\|_{H^3}^2 +\|\sigma\|_{H^2}^2+\|p\|_{H^3}^2\leq e^{C_\sharp(X_0+C_0t)}
  \leq e^{C_\sharp(X_0+C_0\mathcal T_0)}\leq e^{C_\sharp(X_0+1)}=\frac{M_0}{2}.
$$
As a result, recalling that $\mathcal T_\sharp<\mathcal T_\text{max}$, and by the regularities of $(v,\sigma,p)$ on $\mathcal O\times[0,\mathcal T_\text{max})$,
there is some $\mathcal T_\sharp^*\in(\mathcal T_\sharp,\mathcal T_\text{max})$, such that
$$
  p\geq 0.5\underline p,\quad \sigma\geq0.5\underline\sigma,\quad
  (\|v\|_{H^3}^2 +\|\sigma\|_{H^2}^2+\|p\|_{H^3}^2)(t)\leq M_0,
$$
for any $t\in[0,\mathcal T_\sharp^*]$. Recalling the definition of $\mathcal T_\sharp$, this implies
$T_\sharp\geq T_\sharp^*$, contradicting to the statement that $\mathcal T_\sharp^*\in(\mathcal T_\sharp,\mathcal T_\text{max})$.
Therefore, the claim holds.

Thanks to \eqref{T0}, the conclusion follows from \eqref{LBsp}, \eqref{APRESTTIME}, and \eqref{Hkvsp}.
\end{proof}

\section{Proof of the main result}
\label{sec4}

In this section, we give the proof of our main result, Theorem \ref{thmmain}.

\begin{proof}
We first proof the existence.
Denote
$$
X_0:=2(\|v_{0}\|_{H^3}^2+\|\sigma_{0}\|_{H^2}^2+\|p_0\|_{H^3}^2).
$$
Choose $(v_{0,\epsilon}, \sigma_{0,\epsilon}, p_{0,\epsilon})$ such that
\begin{eqnarray*}
  (v_{0,\epsilon}, \sigma_{0,\epsilon})\in H^3(\mathcal O),\quad p_{0,\epsilon}\in H^3(\mathbb T^3),\\ \partial_zv_{0,\epsilon}|_{z=0,1}=0,\quad \partial_z\sigma_{0,\epsilon}|_{z=0,1}=0,\quad
  \sigma_{0,\epsilon}\geq\underline\sigma, \quad p_{0,\epsilon}\geq\underline p,
\end{eqnarray*}
and
$$
  v_{0,\epsilon}\rightarrow v_0,\quad\mbox{in }H^3(\mathcal O);\quad
  \sigma_{0,\epsilon}\rightarrow\sigma_0,\quad\mbox{in }H^2(\mathcal O);\quad
  p_{0,\epsilon}\rightarrow p_0,\quad\mbox{in }H^3(\mathbb T^2).
$$
Then, there is a positive number $\epsilon_0$, such that
$$
\|v_{0,\epsilon}\|_{H^3}^2+\|\sigma_{0,\epsilon}\|_{H^2}^2+\|p_{0,\epsilon}\|_{H^3}^2\leq X_0,
$$
for any $\epsilon\in(0,\epsilon_0)$.

By Proposition \ref{PROP-APR}, there is a positive time $\mathcal T_0$ depending only on $\gamma,$ $\nu,$ $\mu$, $\lambda$,
$\underline\sigma$, $\underline p$, and $X_0$, such that for any $\epsilon\in(0,\epsilon_0)$, system \eqref{EQv}--\eqref{EQp}, subject to \eqref{BC1}--\eqref{IC}, has a unique solution $(v_\epsilon,
\sigma_\epsilon, p_\epsilon)$ on $\mathcal O\times(0,\mathcal T_0)$,
with initial data $(v_{0,\epsilon}, \sigma_{0,\epsilon}, p_{0,\epsilon})$, satisfying
 \begin{eqnarray}
  &&\inf_{(x,y,z,t)\in\mathcal O\times[0,\mathcal T_0]}\sigma_\epsilon\geq0.5\underline\sigma,\quad \inf_{(x,y,t)\in\mathbb T^2\times[0,\mathcal T_0]}p_\epsilon\geq0.5\underline p,\label{APRLBsp}\\
  &&\sup_{0\leq t\leq\mathcal T_0}(\|v_\epsilon\|_{H^3}^2+\|\sigma_\epsilon\|_{H^2}^2+\|p_\epsilon\|_{H^3}^2)(t)
  +\epsilon\int_0^{\mathcal T_0}(\|\nabla_h\sigma_\epsilon\|_{H^2}^2+\|\nabla_hp_\epsilon\|_{H^3}^2)dt\nonumber\\
  &&\, \, ~~~+\int_0^{\mathcal T_0}(\|v_\epsilon\|_{H^4}^2+\|\partial_z\sigma_\epsilon\|_{H^2}^2+\|\partial_tv_\epsilon\|_{H^2}^2+\|\partial_t\sigma_\epsilon
  \|_{H^1}^2+\|\partial_tp_\epsilon\|_{H^2}^2)dt \leq K_0,\label{APRHk}
  \end{eqnarray}
  for a positive constant $K_0$ depending only on $\gamma$, $\nu$, $\mu$, $\lambda$, $\underline\sigma$, $\underline p$, and
  $X_0$.

  Thanks to \eqref{APRHk}, by the Banach-Alaoglu theorem, and using Cantor's diagonal argument,
  there is a subsequence $(v_{\epsilon_n}, \sigma_{\epsilon_n}, p_{\epsilon_n})$ and a triple $(v,\sigma,p)$,
  such that
  \begin{eqnarray}
    &&v_{\epsilon_n}\stackrel{\ast}{\rightharpoonup}v,\quad\mbox{in }L^\infty(0,\mathcal T_0; H^3(\mathcal O)),\label{WeaC1}\\
    &&\sigma_{\epsilon_n}\stackrel{\ast}{\rightharpoonup}\sigma,\quad\mbox{in }L^\infty(0,\mathcal T_0; H^2(\mathcal O)),\label{WeaC2}\\
    &&p_{\epsilon_n}\stackrel{\ast}{\rightharpoonup}p,\quad\mbox{in }L^\infty(0,\mathcal T_0; H^3(\mathbb T^2)),\label{WeaC3}\\
    &&v_{\epsilon_n}{\rightharpoonup}v,\quad\mbox{in }L^2(0,\mathcal T_0; H^4(\mathcal O)),\label{WeaC4}\\
    &&\partial_z\sigma_{\epsilon_n} {\rightharpoonup}\partial_z\sigma,\quad\mbox{in }L^2(0,\mathcal T_0; H^2(\mathcal O)),\label{WeaC5}\\
    &&\partial_tv_{\epsilon_n} {\rightharpoonup}\partial_tv,\quad\mbox{in }L^2(0,\mathcal T_0; H^2(\mathcal O)),\label{WeaC6}\\
    &&\partial_t\sigma_{\epsilon_n} {\rightharpoonup}\partial_t\sigma,\quad\mbox{in }L^2(0,\mathcal T_0; H^1(\mathcal O)),\label{WeaC7}\\
    &&\partial_tp_{\epsilon_n} {\rightharpoonup}\partial_tp,\quad\mbox{in }L^2(0,\mathcal T_0; H^2(\mathbb T^2)),\label{WeaC8}
  \end{eqnarray}
  where $\stackrel{\ast}{\rightharpoonup}$ and $\rightharpoonup$ represent the weak-* and weak convergence, respectively, in the corresponding spaces. With the aid of the above convergence, noticing that $H^3(\mathcal O)\hookrightarrow\hookrightarrow H^2(\mathcal O)$, $H^2(\mathcal O)\hookrightarrow\hookrightarrow H^1(\mathcal O)\cap C(\overline{\mathcal O})$, and $H^3(\mathbb T^2)\hookrightarrow\hookrightarrow H^2(\mathbb T^2)\cap C(\mathbb T^2)$,
  it follows from the Aubin-Lions lemma that
    \begin{eqnarray}
    &&v_{\epsilon_n} \rightarrow v,\quad\mbox{in }C([0,\mathcal T_0]; H^2(\mathcal O)),\label{StrC1}\\
    &&\sigma_{\epsilon_n} \rightarrow\sigma,\quad\mbox{in }C([0,\mathcal T_0]; H^1(\mathcal O)\cap C(\overline{\mathcal O})),\label{StrC2}\\
    &&\partial_z\sigma_{\epsilon_n} \rightarrow\partial_z\sigma,\quad\mbox{in }L^2(0,\mathcal T_0; H^1(\mathcal O)),\\
    &&p_{\epsilon_n} \rightarrow p,\quad\mbox{in }C([0,\mathcal T_0]; H^2(\mathbb T^2)\cap C(\mathbb T^2)).\label{StrC3}
  \end{eqnarray}
  These imply that $(v,\sigma,p)$ has initial data $(v_0,\sigma_0,p_0)$, enjoys all regularities stated in Theorem \ref{thmmain}, and fulfills the boundary conditions \eqref{BC1} and \eqref{BC2}.
  Besides, by \eqref{StrC2} and \eqref{StrC3}, it follows from \eqref{APRLBsp} that
  \begin{equation*}
    \inf_{(x,y,z,t)\in\mathcal O\times[0,\mathcal T_0]}\sigma \geq0.5\underline\sigma,\quad \inf_{(x,y,t)\in\mathbb T^2\times[0,\mathcal T_0]}p \geq0.5\underline p.
  \end{equation*}
  By \eqref{WeaC1}--\eqref{StrC3}, one can take the limit as $n\rightarrow\infty$ to \eqref{EQv}--\eqref{EQp} to show that $(v,\sigma,p)$ satisfies system \eqref{EQv0}--\eqref{EQp0} in the sense of distribution and further  pointwisely, by the regularities of $(v,\sigma,p)$. Therefore, $(v,\sigma,p)$ is a strong solution
  to system \eqref{EQv}--\eqref{EQp}, subject to \eqref{BC1}--\eqref{IC}, on $\mathcal O\times(0,\mathcal T_0)$.

  Now, we prove the continuous dependence on the initial data which in particular implies the uniqueness. Let $(v_i, \sigma_i, p_i)$, $i=1,2,$
  be two solutions to system \eqref{EQv0}--\eqref{EQs0}, subject to \eqref{BC1}--\eqref{BC2}, on $\mathcal O\times(0,\mathcal T)$ for some positive time $\mathcal T$, satisfying
  \begin{equation}
    \inf_{(x,y,z,t)\in\mathcal O\times[0,\mathcal T]}\sigma_i\geq0.5\underline\sigma,\quad\inf_{(x,y,t)\in\mathbb T^2\times[0,\mathcal T]}p_i\geq 0.5\underline p,\label{LBspi}
  \end{equation}
  for two positive numbers $\underline\sigma$ and $\underline p$. Let $Q(\nabla v_i)$, $\phi(v_i, p_i)$, and $w_i$, respectively, be given by \eqref{Q}, \eqref{phi}, and \eqref{EQw0}, corresponding to $(v_i, p_i)$, $i=1,2$. Denote
  $$
  (v,w,\sigma,p)=(v_1-v_2, w_1-w_2, \sigma_1-\sigma_2, p_1-p_2)
  $$
  and set
  \begin{equation}
  \label{M}
  M=\sup_{0\leq t\leq\mathcal T}\left(\|(v_1, v_2)\|_{H^3}^2+\|(\sigma_1,\sigma_2)\|_{H^2}^2+\|(p_1,p_2)\|_{H^3}^2\right)(t).
  \end{equation}
  Then, by the regularities of $(v_i, \sigma_i, p_i)$, $i=1,2,$ $M$ is a finite number.

  In the same way as \eqref{H2phi} and \eqref{H2w}, one has
  \begin{eqnarray}
    \|\phi(v_i,p_i)\|_{H^2}(t)\leq C_M,\quad \|w_i\|_{H^2}(t)\leq C_M(1+\|\partial_z\sigma_i\|_{H^2}(t)),
\label{H2wphii}
  \end{eqnarray}
  for any $t\in[0,\mathcal T]$, where $C_M$ is a positive constant depending on $M$.
  By direct calculations, it holds that
  \begin{eqnarray*}
    &&\phi(v_1,p_1)-\phi(v_2,p_2)\\
    &=&\text{div}_h\tilde v+\frac{1}{\gamma p_1}\left[\tilde v_1\cdot\nabla_hp
    +\tilde v\cdot\nabla_hp_2-(\gamma-1)\left(\widetilde{Q(\nabla v_1)}-\widetilde{Q(\nabla v_2)}\right)\right]\\
    &&+\frac{p}{\gamma p_1p_2}\left((\gamma-1)\widetilde{Q(\nabla v_2)}-\tilde v_2\cdot\nabla_hp_2\right)
  \end{eqnarray*}
  and thus by the Sobolev inequality and \eqref{M} one deduces
  \begin{eqnarray*}
    &&\|\phi(v_1,p_1)-\phi(v_2,p_2)\|_2\\
    &\leq& C(\|\nabla_hv\|_2+\|v_1\|_\infty\|\nabla_hp\|_2+\|v\|_2\|\nabla_hp_2\|_\infty+\|(\nabla v_1,\nabla v_2)\|_\infty\|\nabla v\|_2)\\
    &&+C\|p\|_2(\|\nabla v_2\|_\infty^2+\|v_2\|_\infty\|\nabla_hp_2\|_\infty)\\
    &\leq& C(\|\nabla_hv\|_2+\|v_1\|_{H^2}\|\nabla_hp\|_2+\|v\|_2\|p_2\|_{H^3}+\|(v_1,v_2)\|_{H^3}\|\nabla v\|_2)\\
    &&+C\|p\|_2(\|v_2\|_{H^3}^2+\|v_2\|_{H^2}\|p_2\|_{H^3})\\
    &\leq &C_M(\|v\|_{H^1}+\|p\|_{H^1}),
  \end{eqnarray*}
  that is
  \begin{equation}
    \|\phi(v_1,p_1)-\phi(v_2,p_2)\|_2\leq C_M(\|v\|_{H^1}+\|p\|_{H^1}). \label{L2Dffphi}
  \end{equation}
  Thanks to this, it follows that
  \begin{eqnarray}
    \|w\|_2&=&\|w_1-w_2\|_2=\left\|\nu\partial_z\sigma-\int_0^z(\phi(v_1,p_1)-\phi(v_2,p_2))dz'\right\|_2\nonumber\\
    &\leq&\nu\|\partial_z\sigma\|_2+C_M(\|v\|_{H^1}+\|p\|_{H^1}).\label{L2Diffw}
  \end{eqnarray}

  Direct calculations yield
  \begin{align*}
    \partial_tv+(v_1\cdot\nabla_h)v+w_1\partial_zv+(v\cdot\nabla_h)v_2+w\partial_zv_2+\sigma_1\nabla_hp+\sigma\nabla_hp_2\\
    =\sigma_1(\mu\Delta v+(\mu+\lambda)\nabla_h\text{div}_hv)+\sigma(\mu\Delta v_2+(\mu+\lambda)\nabla_h\text{div}_hv_2).
  \end{align*}
  Multiplying this with $v-\Delta v$, integrating over $\mathcal O$, and noticing that
  \begin{eqnarray*}
    -\int\sigma_1\nabla_h\text{div}_hv\cdot\Delta vdxdydz
    =\int\text{div}_hv(\nabla_h\sigma_1\cdot\Delta v+\sigma_1\Delta\text{div}_hv)dxdydz\\
    =-\int\sigma_1|\nabla\text{div}_hv|^2dxdydz+\int\text{div}_hv(\nabla_h\sigma_1\cdot\Delta v-\nabla\sigma_1\cdot\nabla\text{div}_hv)
    dxdydz,
  \end{eqnarray*}
  one obtains
  \begin{eqnarray}
    &&\frac12\frac{d}{dt}\|v\|_{H^1}^2+\int\sigma_1(\mu|\Delta v|^2+(\mu+\lambda)|\nabla\text{div}_hv|^2)dxdydz\nonumber\\
    &=&\int((v_1\cdot\nabla_h)v+w_1\partial_zv+(v\cdot\nabla_h)v_2+w\partial_zv_2)\cdot(\Delta v-v)dxdydz\nonumber\\
    &&+\int(\sigma_1\nabla_hp+\sigma\nabla_hp_2)\cdot(\Delta v-v)dxdydz\nonumber\\
    &&+\int\sigma_1(\mu\Delta v+(\mu+\lambda)\nabla_h\text{div}_hv)\cdot vdxdydz\nonumber\\
    &&+\int\sigma(\mu\Delta v_2+(\mu+\lambda)\nabla_h\text{div}_hv_2)\cdot(v-\Delta v)dxdydz\nonumber\\
    &&+(\mu+\lambda)\int\text{div}_hv(\nabla_h\sigma_1\cdot\Delta v-\nabla\sigma_1\cdot\nabla\text{div}_hv)dxdydz\nonumber \\
    &=:&I_1+I_2+I_3+I_4+I_5. \label{922-5}
  \end{eqnarray}
  Estimates on $I_i, i=1,2,\cdots,5$, are carried out as follows. For $I_1$, it follows from the H\"older, Sobolev, and Young inequalities, \eqref{M}, and
  \eqref{H2wphii} that
  \begin{eqnarray*}
    I_1&\leq&(\|(v_1,w_1)\|_\infty\|\nabla v\|_2+\|(v,w)\|_2\|\nabla v_2\|_\infty)\|(v,\Delta v)\|_2\\
    &\leq&C[\|(v_1,w_1)\|_{H^2}\|\nabla v\|_2+\|v_2\|_{H^3}(\|v\|_{H^1}+\|p\|_{H^1}+\|\partial_z\sigma\|_2)]\|(v,\Delta v)\|_2\\
    &\leq&C_M[(1+\|\partial_z\sigma_1\|_{H^2})\|v\|_{H^1}+\|p\|_{H^1}+\|\partial_z\sigma\|_2]\|(v,\Delta v)\|_2\\
    &\leq&\frac\eta5\|\Delta v\|_2^2+C_{\eta,M}[(1+\|\partial_z\sigma_1\|_{H^2}^2)\|v\|_{H^1}^2+\|p\|_{H^1}^2+\|\partial_z\sigma\|_2^2],
  \end{eqnarray*}
  for any $\eta>0$. For $I_2, I_3, I_4,$ and $I_5$, it follows from the H\"older, Sobolev, and Young inequalities and \eqref{M} that
  \begin{eqnarray*}
    I_2&\leq&(\|\sigma_1\|_\infty\|\nabla_hp\|_2+\|\sigma\|_2\|\nabla_hp_2\|_\infty)\|(v,\Delta v)\|_2\\
    &\leq&C(\|\sigma_1\|_{H^2}+\|p_2\|_{H^3})(\|\nabla_hp\|_2+\|\sigma\|_2)\|(v,\Delta v)\|_2\\
    &\leq&\frac\eta5\|\Delta v\|_2^2+C_{\eta,M}(\|v\|_2^2+\|\sigma\|_2^2+\|\nabla_hp\|_2^2),\\
    I_3&\leq&C\|\sigma_1\|_\infty\|\nabla^2v\|_2\|v\|_2\leq C\|\sigma_1\|_{H^2}\|\nabla^2v\|_2\|v\|_2
     \leq \frac\eta5\|\nabla^2v\|_2^2+C_{\eta,M}\|v\|_2^2,\\
    I_4&\leq& C\|\sigma\|_2\|\nabla^2v_2\|_\infty\|(v,\Delta v)\|_2
    \leq C\|v_2\|_{H^4}\|\sigma\|_2\|(v,\Delta v)\|_2\\
    &\leq&\frac\eta5\|\Delta v\|_2^2+C_{\eta,M}(1+\|v_2\|_{H^4}^2)(\|v\|_2^2+\|\sigma\|_2^2),\\
    I_5&\leq&C\|\nabla_hv\|_3\|\nabla\sigma_1\|_6\|\nabla^2v\|_2\leq C\|\nabla_hv\|_2^\frac12\|\nabla_h\nabla v\|_2^\frac12\|\sigma_1\|_{H^2}\|\nabla^2v\|_2\\
    &\leq&\frac\eta5\|\nabla^2v\|_2^2+C_{\eta,M}\|\nabla v\|_2^2,
  \end{eqnarray*}
  for any $\eta>0$. Substituting the estimate $I_i, i=1,2,\cdots,5$, into \eqref{922-5} and recalling \eqref{LBspi}, one gets
  \begin{eqnarray}
    &&\frac12\frac{d}{dt}\|v\|_{H^1}^2+0.5\mu\underline\sigma\|\Delta v\|_2^2\nonumber\\
    &\leq& C_{\eta,M}(1+\|\partial_z\sigma_1\|_{H^2}^2+\|v_2\|_{H^4}^2)(\|v\|_{H^1}^2+\|\sigma\|_2^2+\|p\|_{H^1}^2)\nonumber\\
    &&+\eta\|\nabla^2v\|_2^2+C_{\eta,M}\|\partial_z\sigma\|_2^2,\label{922-6}
  \end{eqnarray}
  for any $\eta>0$. By the elliptic estimate, it holds that
  $$
  \|\nabla^2v\|_2\leq C(\|\Delta v\|_2+\|v\|_2).
  $$
  Substituting this into \eqref{922-6} and choosing $\eta$ sufficiently small lead to
  \begin{eqnarray}
    && \frac{d}{dt}\|v\|_{H^1}^2+0.5\mu\underline\sigma\|\Delta v\|_2^2\nonumber\\
    &\leq& C_M^*\|\partial_z\sigma\|_2^2+
    C_{M}^*(1+\|\partial_z\sigma_1\|_{H^2}^2+\|v_2\|_{H^4}^2)(\|v\|_{H^1}^2+\|\sigma\|_2^2+\|p\|_{H^1}^2),\label{922-7}
  \end{eqnarray}
  for a positive constant $C_M^*$ depending on $M$.

  By direct calculations, one has
  \begin{align*}
    \partial_t\sigma+&v_1\cdot\nabla_h\sigma+w_1\partial_z\sigma+v\cdot\nabla_h\sigma_2+w\partial_z\sigma_2+\sigma_1(\phi(v_1,p_1)-\phi(v_2,p_2)
    -\text{div}_hv)\\
    &+\sigma(\phi(v_2,p_2)-\text{div}_h v_2)=\nu\sigma_1\partial_z^2\sigma+\nu\sigma\partial_z^2\sigma_2.
  \end{align*}
  Multiplying this with $\sigma$, integrating over $\mathcal O$, and noticing that
  \begin{eqnarray*}
    &-\int v_1\cdot\nabla_h\sigma\sigma dxdydz=\frac12\int\nabla_h\cdot v_1\sigma^2dxdydz,\\
    &-\int\sigma_1\partial_z^2\sigma\sigma dxdydz=\int\sigma_1|\partial_z\sigma|^2dxdydz+\int\partial_z\sigma_1\partial_z\sigma\sigma dxdydz,\\
    &\int\partial_z^2\sigma_2\sigma^2dxdydz=-2\int\partial_z\sigma_2\sigma\partial_z\sigma dxdydz,
  \end{eqnarray*}
  one gets
  \begin{eqnarray*}
    &&\frac12\frac{d}{dt}\|\sigma\|_2^2+\nu\int\sigma_1|\partial_z\sigma|^2dxdydz\\
    &=&\frac12\int\nabla_h\cdot v_1\sigma^2dxdydz-\int(w_1\partial_z\sigma+v\cdot\nabla_h\sigma_2+w\partial_z\sigma_2)\sigma dxdydz\\
    &&-\int\sigma_1(\phi(v_1,p_1)-\phi(v_2,p_2)-\text{div}_hv)\sigma dxdydz\\
    &&-\int(\phi(v_2,p_2)-\text{div}_hv_2)\sigma^2dxdydz\\
    &&-\nu\int(2\partial_z\sigma_2\sigma\partial_z\sigma+\partial_z\sigma_1\partial_z\sigma\sigma)dxdydz.
  \end{eqnarray*}
  Thanks to this, recalling \eqref{LBspi}, and using \eqref{H2wphii}, \eqref{L2Dffphi}, and \eqref{L2Diffw}, one deduces by the H\"older,
  Sobolev, and Young inequalities and \eqref{M} that
  \begin{eqnarray*}
     &&\frac12\frac{d}{dt}\|\sigma\|_2^2+0.5\nu\underline\sigma\|\partial_z\sigma\|_2^2\\
     &\leq&C[\|\nabla_hv_1\|_\infty\|\sigma\|_2^2+\|w_1\|_\infty\|\partial_z\sigma\|_2\|\sigma\|_2+\|v\|_6\|\nabla_h\sigma_2\|_3\|\sigma\|_2
     +\|w\|_2\|\partial_z\sigma_2\|_\infty\|\sigma\|_2\\
     &&+\|\sigma_1\|_\infty\|\phi(v_1,p_1)-\phi(v_2,p_2)\|_2\|\sigma\|_2+\|\sigma_1\|_\infty\|\text{div}_hv\|_2\|\sigma\|_2
     \\
     &&+(\|\phi(v_2,p_2)\|_\infty+\|\nabla_hv_2\|_\infty)\|\sigma\|_2^2+(\|\partial_z\sigma_2\|_\infty+\|\partial_z\sigma_1\|_\infty)
     \|\partial_z\sigma\|_2\|\sigma\|_2]\\
     &\leq&C[\|v_1\|_{H^3}\|\sigma\|_2^2+\|w_1\|_{H^2}\|\partial_z\sigma\|_2\|\sigma\|_2+\|v\|_{H^1}\|\sigma_2\|_{H^2}\|\sigma\|_2\\
     &&+(\|\partial_z\sigma\|_2+\|v\|_{H^1}+\|p\|_{H^1})\|\partial_z\sigma_2\|_{H^2}\|\sigma\|_2+\|\sigma_1\|_{H^2}
     (\|v\|_{H^1}+\|p\|_{H^1})\|\sigma\|_2\\
     &&+\|\sigma_1\|_{H^2}\|v\|_{H^1}\|\sigma\|_2+(\|\phi(v_2,p_2)\|_{H^2}+\|v_2\|_{H^3})\|\sigma\|_2^2\\
     &&+\|(\partial_z\sigma_1,\partial_z\sigma_2)\|_{H^2}\|\partial_z\sigma\|_2\|\sigma\|_2]\\
     &\leq&C_M[\|\sigma\|_2^2+(1+\|\partial_z\sigma_1\|_{H^2})\|\partial_z\sigma\|_2\|\sigma\|_2+\|v\|_{H^1}\|\sigma\|_2\\
     &&+\|\partial_z\sigma_2\|_{H^2}\|\sigma\|_2(\|\partial_z\sigma\|_2+\|v\|_{H^1}+\|p\|_{H^1})+(\|v\|_{H^1}+\|p\|_{H^1})\|\sigma\|_2]\\
     &\leq&0.25\nu\underline\sigma\|\partial_z\sigma\|_2^2+C_{M}(1+\|(\partial_z\sigma_1,\partial_z\sigma_2)\|_{H^2}^2)
     (\|v\|_{H^1}^2+\|\sigma\|_2^2+\|p\|_{H^1}^2)
  \end{eqnarray*}
  and thus
  \begin{eqnarray}
     &&\frac{d}{dt}\|\sigma\|_2^2+0.5\nu\underline\sigma\|\partial_z\sigma\|_2^2\nonumber\\
     &\leq&C_{M}(1+\|(\partial_z\sigma_1,\partial_z\sigma_2)\|_{H^2}^2)
     (\|v\|_{H^1}^2+\|\sigma\|_2^2+\|p\|_{H^1}^2). \label{922-8}
  \end{eqnarray}

  Direct calculations yield
  \begin{align*}
    \partial_tp+\bar v_1\cdot\nabla_hp+\bar v\cdot\nabla_hp_2+\gamma p_1\text{div}_h\bar v+\gamma p\text{div}_h\bar v_2
    =(\gamma-1)\left(\overline{Q(\nabla v_1)}-\overline{Q(\nabla v_2)}\right).
  \end{align*}
  Multiplying this with $p-\Delta_hp$ and integrating over $\mathbb T^2$ yield
  \begin{eqnarray}
    \frac12\frac{d}{dt}\|p\|_{H^1}^2&=&\int\bar v_1\cdot\nabla_hp\Delta_hpdxdydz\nonumber\\
    &&+\int(\bar v\cdot\nabla_hp_2+\gamma p_1\text{div}_h\bar v+\gamma p\text{div}_h\bar v_2)\Delta_hpdxdydz\nonumber\\
    &&-(\gamma-1)\int \left(\overline{Q(\nabla v_1)}-\overline{Q(\nabla v_2)}\right)\Delta_hp dxdydz\nonumber\\
    &&-\int(\bar v_1\cdot\nabla_hp+\bar v\cdot\nabla_hp_2+\gamma p_1\text{div}_h\bar v+\gamma p\text{div}_h\bar v_2)pdxdydz\nonumber\\
    &&+(\gamma-1)\int \left(\overline{Q(\nabla v_1)}-\overline{Q(\nabla v_2)}\right)pdxdydz\nonumber\\
    &=:&II_1+II_2+II_3+II_4+II_5. \label{AD1}
  \end{eqnarray}
  Estimates on $II_i$, $i=1,2,\cdots,5$, are carried out as follows. It follows from integrating by parts, the H\"older, Sobolev, and Young inequalities and \eqref{M} that
  \begin{eqnarray*}
    II_1&=&\int\bar v_1\cdot\nabla_hp\Delta_hpdxdydz\nonumber\\
    &=&-\int[\nabla_hp\cdot\nabla_h\bar v_1\cdot\nabla_hp+(\bar v_1\cdot\nabla_h)\nabla_hp
    \cdot\nabla_hp]dxdydz\\
    &=&\frac12\int\text{div}_h\bar v_1|\nabla_hp|^2dxdydz-\int\nabla_hp\cdot\nabla_h\bar v_1\cdot\nabla_hpdxdydz \\
    &\leq&C\|\nabla v_1\|_\infty\|\nabla_hp\|_2^2\leq C\|v_1\|_{H^3}\|\nabla_hp\|_2^2\leq C_M\|p\|_{H^1}^2, \\
    II_2&=&-\int\nabla_h(\bar v\cdot\nabla_hp_2+\gamma p_1\text{div}_h\bar v+\gamma p\text{div}_h\bar v_2)\cdot\nabla_hpdxdydz\\
    &\leq&\int(|\nabla_h\bar v||\nabla_hp_2|+|\bar v||\nabla_h^2p_2|+\gamma|\nabla_hp_1||\text{div}_h\bar v|+\gamma|p_1|
    |\nabla_h\text{div}_h\bar v|\\
    &&+\gamma|\nabla_hp||\text{div}_h\bar v_2|+\gamma|p||\nabla_h\text{div}_h\bar v_2|)|\nabla_hp|dxdydz\\
    &\leq&C(\|(\nabla_hp_1,\nabla_hp_2)\|_\infty\|\nabla_hv\|_2+\|v\|_\infty\|\nabla_h^2p_2\|_2+\|p_1\|_\infty\|\nabla_h^2v\|_2\\
    &&+\|\nabla_hv_2\|_\infty\|\nabla_hp\|_2+\|p\|_2\|\nabla_h^2v_2\|_\infty)\|\nabla_hp\|_2\\
    &\leq&C(\|(p_1,p_2)\|_{H^3}\|\nabla_hv\|_2+\|v\|_{H^2}\|p_2\|_{H^2}+\|p_1\|_{H^2}\|v\|_{H^2}\\
    &&+\|v_2\|_{H^3}\|\nabla_hp\|_2+\|v_2\|_{H^4}\|p\|_2)\|\nabla_hp\|_2\\
    &\leq&\frac\eta4\|v\|_{H^2}^2+C_{\eta,M}(1+\|v_2\|_{H^4}^2)(\|v\|_{H^1}^2+\|p\|_{H^1}^2),\\
    II_3&=&(\gamma-1)\int \nabla_h\left(\overline{Q(\nabla v_1)}-\overline{Q(\nabla v_2)}\right)\cdot\nabla_hp dxdydz\\
    &\leq&C\int(|\nabla(v_1+v_2)||\nabla^2v|+|\nabla^2(v_1+v_2)||\nabla v|)|\nabla_hp|dxdydz\\
    &\leq&C(\|\nabla(v_1+v_2)\|_\infty\|\nabla^2v\|_2+\|\nabla^2(v_1+v_2)\|_\infty\|\nabla v\|_2)\|\nabla_hp\|_2\\
    &\leq&C(\|v_1+v_2\|_{H^3}\|\nabla^2v\|_2+\|v_1+v_2\|_{H^4}\|\nabla v\|_2)\|\nabla_hp\|_2\\
    &\leq&\frac\eta4\|\nabla^2v\|_2^2+C_\eta\|v_1+v_2\|_{H^3}^2\|\nabla_hp\|_2^2+\|v_1+v_2\|_{H^4}(\|v\|_{H^1}^2+\|p\|_{H^1}^2)\\
    &\leq&\frac\eta4\|\nabla^2v\|_2^2+C_{\eta,M}\left(1+\|(v_1,v_2)\|_{H^4}^2\right)(\|v\|_{H^1}^2+\|p\|_{H^1}^2),
  \end{eqnarray*}
  for any $\eta>0$.
  By the H\"older, Sobolev, and Young inequalities and using \eqref{M}, one deduces
  \begin{eqnarray*}
    II_4
    &\leq&C(\|\bar v_1\|_\infty\|\nabla_hp\|_2+\|\bar v\|_2\|\nabla_hp_2\|_\infty+\|p_1\|_\infty\|\text{div}_h\bar v\|_2+
    \|p\|_2\|\text{div}_h\bar v_2\|_\infty)\|p\|_2\\
    &\leq&C(\|v_1\|_{H^2}\|\nabla_hp\|_2+\|v\|_2\|p_2\|_{H^3}+\|p_1\|_{H^2}\|\nabla_hv\|_2+\|p\|_2\|v_2\|_{H^3})\|p\|_2\\
    &\leq&C_M(\|p\|_{H^1}^2+\|v\|_{H^1}^2)
  \end{eqnarray*}
  and
  \begin{eqnarray*}
    II_5&\leq&C\int|\nabla(v_1+v_2)||\nabla v||p|dxdydz\leq C\|\nabla(v_1+v_2)\|_\infty\|\nabla v\|_2\|p\|_2\\
    &\leq&C\|(v_1,v_2)\|_{H^3}\|v\|_{H^1}\|p\|_2\leq C_M(\|p\|_2^2+\|v\|_{H^1}^2).
  \end{eqnarray*}
  Substituting the estimates for $II_i, i=1,2,\cdots,5$, into \eqref{AD1} and by the elliptic estimate, one gets
    \begin{eqnarray}
    \frac{d}{dt}\|p\|_{H^1}^2
    &\leq&\eta\|v\|_{H^2}^2+C_{\eta,M}\left(1+\|(v_1,v_2)\|_{H^4}^2\right)(\|p\|_{H^1}^2+\|v\|_{H^1}^2)\nonumber\\
    &\leq&\eta C_*(\|\Delta v\|_2^2+\|v\|_2^2)+C_{\eta,M}\left(1+\|(v_1,v_2)\|_{H^4}^2\right)(\|p\|_{H^1}^2+\|v\|_{H^1}^2)\nonumber\\
    &\leq&\eta C_*\|\Delta v\|_2^2+C_{\eta,M}\left(1+\|(v_1,v_2)\|_{H^4}^2\right)(\|p\|_{H^1}^2+\|v\|_{H^1}^2),\label{922-9}
  \end{eqnarray}
  for any $\eta>0$ and for a positive constant $C_*$.

  Let $C_M^*$ and $C_*$ be the constants in \eqref{922-7} and \eqref{922-9}, respectively, and set
  $$
  \delta_*=\frac{\nu\underline\sigma}{4 C_M^*},\quad \eta_*=\frac{\mu\nu\underline\sigma^2}{16C_M^*C_*}.
  $$
  Multiplying \eqref{922-7} with $\delta_*$, summing the resultant with \eqref{922-8} and \eqref{922-9}, and choosing $\eta=\eta_*$, then one obtains
  \begin{eqnarray*}
    &&\frac{d}{dt}(\|\delta_*\|v\|_{H^1}^2+\|\sigma\|_2^2+\|p\|_{H^1}^2)+\frac{\mu\nu\underline\sigma^2}{16C_M^*}\|\Delta v\|_2^2
    +0.25\nu\underline\sigma\|\partial_z\sigma\|_2^2\\
    &\leq& C_M(1+\|(\partial_z\sigma_1,\partial_z\sigma_2)\|_{H^2}^2+\|(v_1,v_2)\|_{H^4}^2)(\|v\|_{H^1}^2+\|\sigma\|_2^2+\|p\|_{H^1}^2),
  \end{eqnarray*}
  from which, by the Gr\"onwall inequality, one gets
  \begin{eqnarray*}
    &&\left(\|v\|_{H^1}^2+\|\sigma\|_2^2+\|p\|_{H^1}^2\right)(t)+\int_0^t(\|\Delta v\|_2^2+\|\partial_z\sigma\|_2^2)ds\\
    &\leq& C\exp\left\{C_M\int_0^t(1+\|(\partial_z\sigma_1,\partial_z\sigma_2)\|_{H^2}^2+\|(v_1,v_2)\|_{H^4}^2)ds\right\}\\
    &&\times(\|v_0\|_{H^1}^2
    +\|\sigma_0\|_2^2+\|p_0\|_{H^1}^2),
  \end{eqnarray*}
  which proves the continuous dependence on the initial data.
\end{proof}

\section{Appendix A: proof of \eqref{CAL} and \eqref{COME}}
\begin{proof}[\textbf{Proof of \eqref{CAL} and \eqref{COME}}]
Direct calculations yield
\begin{equation}
  D^\alpha(fg)=\sum_{\beta\leq\alpha}C_\alpha^\beta D^\beta fD^{\alpha-\beta}g\label{1028-0}
\end{equation}
and
\begin{equation}
 [D^\alpha, f]g=D^\alpha(fg)-fD^\alpha g=\sum_{0\not=\beta\leq\alpha}C_\alpha^\beta D^\beta fD^{\alpha-\beta}g,\label{1028-00}
\end{equation}
where $\beta\leq\alpha$ means that all components of $\beta$ are less or equal than the corresponding ones of $\alpha$.

Take arbitrary $\beta\leq\alpha$ and choose $q_1, q_2\in[1,\infty]$ satisfying
\begin{equation}
  \frac{1}{q_1}=\frac1{r_1}\left(1-\frac{|\beta|}{m}\right)+\frac1{s_2}\frac{|\beta|}m,\quad
 \frac{1}{q_2}=\frac1{r_2}\left(1-\frac{|\alpha-\beta|}{m}\right)+\frac1{s_1}\frac{|\alpha-\beta|}m. \label{1028-1}
\end{equation}
Then, recalling that $|\alpha|=m$ and $\beta\leq\alpha$, one has
\begin{eqnarray*}
 \frac{1}{q_1}-\frac{|\beta|}n&=&\frac1{r_1}\left(1-\frac{|\beta|}{m}\right)+\left(\frac1{s_2}-\frac mn\right)\frac{|\beta|}m,\\
 \frac{1}{q_2}-\frac{|\alpha-\beta|}n&=&\frac1{r_2}\left(1-\frac{|\alpha-\beta|}{m}\right)+\left(\frac1{s_1}-\frac m n\right)\frac{|\alpha-\beta|}m\\
 &=&\frac1{r_2} \frac{|\beta|}{m}+\left(\frac1{s_1}-\frac m n\right)\left(1-\frac{|\beta|}{m}\right),
\end{eqnarray*}
and, thus, by the Gagliado-Nirenberg inequality, it holds that
\begin{equation}
  \label{1028-4}
  \|D^\beta f\|_{q_1}\leq C\|f\|_{r_1}^{1-\frac{|\beta|}m}\|f\|_{W^{m,s_2}}^{\frac{|\beta|}m}, \quad  \|D^{\alpha-\beta}g\|_{q_2}\leq C\|g\|_{r_2}^{\frac{|\beta|}m}\|g\|_{W^{m,s_1}}^{1-\frac{|\beta|}m}.
\end{equation}
Recalling $|\alpha|=m$, $\beta\leq\alpha$, and using \eqref{CDT}, one obtains from \eqref{1028-1} that
\begin{eqnarray}
  \frac1{q_1}+\frac1{q_2}=\left(\frac1{r_1}+\frac1{s_1}\right)\left(1-\frac{|\beta|}m\right)+\left(\frac1{s_2}+\frac1{r_2}\right)\frac{|\beta|}m=\frac1q. \label{1028-3}
\end{eqnarray}
Combing \eqref{1028-4} with \eqref{1028-3}, it follows from the H\"older and Young inequalities that
\begin{eqnarray*}
  \|D^\beta fD^{\alpha-\beta}g\|_q&\leq& \|D^\beta f\|_{q_1}\|D^{\alpha-\beta}g\|_{q_2}
  \leq C \|f\|_{r_1}^{1-\frac{|\beta|}m}\|f\|_{W^{m,s_2}}^{\frac{|\beta|}m} \|g\|_{r_2}^{\frac{|\beta|}m}\|g\|_{W^{m,s_1}}^{1-\frac{|\beta|}m}\nonumber\\
  &\leq&C\left(\|f\|_{r_1}\|g\|_{W^{m, s_1}}+\|g\|_{r_2}\|f\|_{W^{m,s_2}}\right),
\end{eqnarray*}
that is
\begin{equation}
  \label{1028-5}
  \|D^\beta fD^{\alpha-\beta}g\|_q \leq C\left(\|f\|_{r_1}\|g\|_{W^{m, s_1}}+\|g\|_{r_2}\|f\|_{W^{m,s_2}}\right),\quad\forall\beta\leq\alpha, |\alpha|=m.
\end{equation}
Thanks to this and recalling \eqref{1028-0}, one obtains \eqref{CAL}.

Take arbitrary $0\not=\beta=(\beta_1,\beta_2,\beta_3)\leq\alpha$. Since $\beta\not=0$, it follows that $\beta_i\geq1$ for some $i\in\{1,2,3\}$ and, thus,
$\alpha_i\geq\beta_i\geq1$.
Without loss of generality assume that $\alpha_1\geq\beta_1\geq1$ (other cases can be dealt with similarly).
Denote $\tilde\alpha=(\alpha_1-1,\alpha_2,\alpha_3)$ and $\tilde\beta=(\beta_1-1,\beta_2,\beta_3)$. It is clear that $|\tilde\alpha|=m-1$ and $\tilde\alpha\geq\tilde\beta$. Then,
by \eqref{1028-5}, one has
\begin{eqnarray}
  \|D^\beta fD^{\alpha-\beta}g\|_q&=&\|D^{\tilde\beta}\partial_1fD^{\tilde\alpha-\tilde\beta}g\|_q \nonumber\\
  &\leq& C\left(\|\partial_1f\|_{r_1}\|g\|_{W^{m-1, s_1}}+\|g\|_{r_2}\|\partial_1f\|_{W^{m-1,s_2}}\right)\nonumber
\end{eqnarray}
and thus
\begin{equation}
  \label{1028-6}
  \|D^\beta fD^{\alpha-\beta}g\|_q \leq C\left(\|\nabla f\|_{r_1}\|g\|_{W^{m, s_1}}+\|g\|_{r_2}\|\nabla f\|_{W^{m-1,s_2}}\right),
\end{equation}
for any $0\not=\beta\leq\alpha$ with $|\alpha|=m$. Combining \eqref{1028-6} with \eqref{1028-00} leads to \eqref{COME}.
\end{proof}

\section{Appendix B: local existence of the regularized system}
\label{APPENDIXB}
Denote
\begin{eqnarray*}
  N_1(v,\sigma,p):=-[(v\cdot\nabla_h)v+w\partial_zv+\sigma\nabla_hp],\label{N1}\\
  N_2(v,\sigma,p):=-[v\cdot\nabla_h\sigma+w\partial_z\sigma+\sigma(\phi(v,p)-\text{div}_hv)],\label{N2}\\
  N_3(v,p):=-(\bar v\cdot\nabla_hp+\gamma p\text{div}_h\bar v)+(\gamma-1)\overline{Q(\nabla v)},\label{N3}
\end{eqnarray*}
where $w, \phi(v,p), \bar f,$ and $Q(\nabla v)$ are expressed, respectively, by\eqref{EQw00'}, \eqref{phi}, \eqref{TBf}, and \eqref{Q}.
Recall $\mathcal O=\mathbb T^2\times(0,1)$. Define a solution mapping
\begin{equation}\label{DEFF}
(v,\sigma,p)\mapsto(V,\Sigma,P)=:\mathfrak{F}(v,\sigma,p),
\end{equation}
where $(V,\Sigma,P)$ is the unique solution to the following linear system
\begin{eqnarray}
  \partial_tV-\mu\sigma\Delta V-(\mu+\lambda)\sigma\nabla_h\text{div}_hV&=&N_1(v,\sigma,p),\label{AEQV}\\
  \partial_t\Sigma-\nu\sigma\partial_z^2\Sigma-\epsilon\Delta_h\Sigma&=&N_2(v,\sigma,p),\label{AEQSIGMA}\\
  \partial_tP-\epsilon\Delta_hP&=&N_3(v,p), \label{AEQP}
\end{eqnarray}
in $\mathcal O\times(0,T)$, subject to
\begin{eqnarray}
  &V,\Sigma, P\text{ are periodic in }x,y, \quad\partial_zV|_{z=0,1}=0,\quad \partial_z\Sigma|_{z=0,1}=0,\label{ABC}\\
  &(V,\Sigma,P)|_{t=0}=(v_0,\sigma_0,p_0),\label{AIC}
\end{eqnarray}
where $(v_0,\sigma_0)\in H^3(\mathcal O)$, $p_0\in H^3(\mathbb T^2),$ and
\begin{equation*}
  \partial_zv_0|_{z=0,1}=0,\quad \partial_z\sigma_0|_{z=0,1}=0,\quad\sigma_0\geq\underline\sigma, \quad p_0\geq\underline p,\label{A1}
\end{equation*}
for two positive numbers $\underline\sigma$ and $\underline p$.

It will be shown in this section that there is a unique fixed point to the map $\mathfrak F$,
which is the unique local solution
to system (\ref{EQv})--(\ref{EQp}), subject to (\ref{BC1})--(\ref{IC}).

Set
\begin{align*}
\mathscr{X}_{T}:=\big\{&f|f\in L^\infty(0,T; H^3(\mathbb T^2))\cap  L^2(0,T; H^4(\mathbb T^2)),\partial_tf\in L^2(0,T;H^2(\mathbb T^2))\big\},\\
\mathscr{Y}_T:=\big\{&f|f\in L^\infty(0,T; H^3(\mathcal O))\cap  L^2(0,T; H^4(\mathcal O)),\\
&\partial_tf\in L^2(0,T;H^2(\mathcal O)),\partial_zf|_{z=0,1}=0\big\},\\
\mathscr{Z}_T:=\big\{&f|f\in  L^\infty(0,T; H^3(\mathbb T^3))\cap L^2(0,T; H^4(\mathbb T^3)),\partial_tf\in L^2(0,T;H^2(\mathbb T^3))\big\}.
\end{align*}

\subsection{Some calculations on $N_i$, $i=1,2,3$}
\begin{proposition}
  \label{PROPA1-1}
   Let $p_i\in\mathscr X_T$ and $(v_i, \sigma_i)\in\mathscr Y_T$, $i=1,2$, satisfying $p_i\geq0.5\underline p$ and
   \begin{eqnarray*}
   \|(v_i,\sigma_i)\|_{L^\infty(0,T; H^3(\mathcal O))}+\|p_i\|_{L^\infty(0,T;H^3(\mathbb T^2))}\leq M,
   \end{eqnarray*}
   for a positive constant $M$. Then, there is a positive constant $C$, depending only on $M, \underline p, \mu, \lambda, \nu, \gamma$, such that
   \begin{eqnarray*}
     &&\|N_j(v_1,\sigma_1,p_1)-N_j(v_2,\sigma_2,p_2)\|_{L^\infty(0,T; H^2(\mathcal O))}\\
     &\leq&C\left(\|(v_1-v_2,\sigma_1-\sigma_2)\|_{L^\infty(0,T;H^3(\mathcal O))}+\|p_1-p_2\|_{L^\infty(0,T;H^3(\mathbb T^2))}\right),\quad j=1,2,\\
     &&\|N_3(v_1,p_1)-N_3(v_2,p_2)\|_{L^\infty(0,T; H^2(\mathbb T^2))}\\
     &\leq&C\left(\|v_1-v_2\|_{L^\infty(0,T;H^3(\mathcal O))}+\|p_1-p_2\|_{L^\infty(0,T;H^3(\mathbb T^2))}\right),
   \end{eqnarray*}
   and
   $$
   \sum_{j=1}^2\|N_j(v_i,\sigma_i,p_i)\|_{L^\infty(0,T; H^2(\mathcal O))}+\|N_3(v_i,p_i)\|_{L^\infty(0,T; H^2(\mathbb T^2))}\leq C.
   $$
\end{proposition}

\begin{proof}
Noticing that
\begin{eqnarray*}
  \phi(v_1,p_1)-\phi(v_2,p_2)
   = \text{div}_h(\tilde v_1-\tilde v_2)+\frac{1}{\gamma p_1}\Big[\tilde v_1\cdot\nabla_h(p_1-p_2)+(\tilde v_1-\tilde v_2)\cdot\nabla_hp_2\\
  -(\gamma-1)\left(\widetilde{Q(\nabla v_1)}-\widetilde{Q(\nabla v_2)}\right)\Big]+\frac{p_1-p_2}{\gamma  p_1p_2}\left(\tilde v_2\cdot\nabla_hp_2-(\gamma-1)\widetilde{Q(\nabla v_2)}\right),
  \end{eqnarray*}
it follows from \eqref{ALGHk} and \eqref{H21P} that
\begin{eqnarray}
  &&\|\phi(v_1,p_1)-\phi(v_2,p_2)\|_{H^2}\nonumber\\
  &\leq&C\|v_1-v_2\|_{H^3}+C\left\|\frac{1}{p_1}\right\|_{H^2}[\|v_1\|_{H^2}\|p_1-p_2\|_{H^3}+\|v_1-v_2\|_{H^2}\|p_2\|_{H^3}\nonumber\\
  &&+\|(\nabla v_1,\nabla v_2)\|_{H^2}\|\nabla v_1-\nabla v_2\|_{H^2}]+C\left\|\frac{1}{p_1}\right\|_{H^2}\left\|\frac{1}{p_2}\right\|_{H^2}\|p_1-p_2\|_{H^2}\nonumber\\
  &&\times(\|v_2\|_{H^2}\|\nabla p_2\|_{H^2}+\|\nabla v_2\|_{H^2}^2)\nonumber\\
  &\leq&C(\|v_1-v_2\|_{H^3}+\|p_1-p_2\|_{H^3}).\label{ESTPHID}
\end{eqnarray}
Therefore,
\begin{eqnarray}
  \|w_1-w_2\|_{H^2}&\leq&C(\|\sigma_1-\sigma_2\|_{H^3}+\|\phi(v_1,p_1)-\phi(v_2,p_2)\|_{H^2})\nonumber\\
  &\leq&C(\|\sigma_1-\sigma_2\|_{H^3}+\|v_1-v_2\|_{H^3}+\|p_1-p_2\|_{H^3}).\label{ESTWD}
\end{eqnarray}
Then, noticing that $\phi(0,p_i)=0$, it follows that
\begin{eqnarray}
  \|\phi(v_i,p_i)\|_{H^2}=\|\phi(v_i,p_i)-\phi(0,p_i)\|_{H^2}\leq C\|v_i\|_{H^3}\leq C, \label{ESTPHI}
\end{eqnarray}
and
\begin{eqnarray}
  \|w_i\|_{H^2}\leq C\|\partial_z\sigma_i\|_{H^2}+C\|\phi(v_i,p_i)\|_{H^2}\leq C. \label{ESTW}
\end{eqnarray}
Thanks to \eqref{ESTWD} and \eqref{ESTW}, and using \eqref{ALGHk}, one deduces that
\begin{eqnarray*}
  &&\|N_1(v_1,\sigma_1,p_1)-N_2(v_2,\sigma_2,p_2)\|_{H^2}\nonumber\\
  &\leq&C(\|v_1\cdot\nabla_h(v_1-v_2)\|_{H^2}+\|(v_1-v_2)\cdot\nabla_hv_2\|_{H^2}+\|w_1\partial_z(v_1-v_2)\|_{H^2}\nonumber\\
  &&+\|(w_1-w_2)\partial_zv_2\|_{H^2}+\|\sigma_1\nabla_h(p_1-p_2)\|_{H^2}+\|(\sigma_1-\sigma_2)\nabla_hp_2\|_{H^2}\nonumber\\
  &\leq&C(\|v_1\|_{H^2}\|v_1-v_2\|_{H^3}+\|v_1-v_2\|_{H^2}\|v_2\|_{H^3}+\|w_1\|_{H^2}\|v_1-v_2\|_{H^3}\nonumber\\
  &&+\|w_1-w_2\|_{H^2}\|v_2\|_{H^3}+\|\sigma_1\|_{H^2}\|p_1-p_2\|_{H^3}+\|\sigma_1-\sigma_2\|_{H^2}\|p_2\|_{H^3})\nonumber\\
  &\leq&C(\|v_1-v_2\|_{H^3}+\|\sigma_1-\sigma_2\|_{H^3}+\|p_1-p_2\|_{H^2}).
\end{eqnarray*}
With the aid of this and noticing that $N_1(0,\sigma_i,0.5\underline p)=0$, it holds that
\begin{eqnarray*}
  \|N_1(v_i,\sigma_i,p_i)\|_{H^2}&=&\|N_1(v_i,\sigma_i,p_i)-N_1(0,\sigma_i,0.5\underline p)\|_{H^2}\nonumber\\
  &\leq&C(\|v_i\|_{H^3}+\|p_i-0.5\underline p\|_{H^3})\leq C.
\end{eqnarray*}
Combining \eqref{ESTPHID}--\eqref{ESTW} and using \eqref{ALGHk}, one computes
\begin{eqnarray*}
  &&\|N_2(v_1,\sigma_1,p_1)-N_2(v_2,\sigma_2,p_2)\|_{H^2}\\
  &\leq&\|v_1\cdot\nabla_h(\sigma_1-\sigma_2)\|_{H^2}+\|(v_1-v_2)\cdot\nabla_h\sigma_2\|_{H^2}+\|w_1\partial_z(\sigma_1-\sigma_2)\|_{H^2}\\
  &&+\|(w_1-w_2)\partial_z\sigma_2\|_{H^2}+\|\sigma_1(\phi(v_1,p_1)-\phi(v_2,p_2))\|_{H^2}+\|(\sigma_1-\sigma_2)\phi(v_2,p_2)\|_{H^2}\\
  &&+\|\sigma_1\text{div}_h(v_1-v_2)\|_{H^2}+\|(\sigma_1-\sigma_2)\text{div}_hv_2\|_{H^2}\\
   &\leq&C(\|v_1\|_{H^2}\|\sigma_1-\sigma_2\|_{H^3}+\|v_1-v_2\|_{H^2}\|\sigma_2\|_{H^3}+\|w_1\|_{H^2}\|\sigma_1-\sigma_2\|_{H^3}\\
  &&+\|w_1-w_2\|_{H^2}\|\sigma_2\|_{H^3}+\|\sigma_1\|_{H^2}\|\phi(v_1,p_1)-\phi(v_2,p_2)\|_{H^2}\\
  &&+\|\sigma_1-\sigma_2\|_{H^2}\|\phi(v_2,p_2)\|_{H^2}+\|\sigma_1\|_{H^2}\|v_1-v_2\|_{H^3}+\|\sigma_1-\sigma_2\|_{H^2}\|v_2\|_{H^3})\\
   &\leq&C(\|v_1-v_2\|_{H^3}+\|\sigma_1-\sigma_2\|_{H^3}+\|p_1-p_2\|_{H^2}).
\end{eqnarray*}
Using this and noticing that $N_2(0,0.5\underline\sigma,p_i)=0$, one gets
\begin{eqnarray*}
  \|N_2(v_i,\sigma_i,p_i)\|_{H^2}&=&\|N_2(v_i,\sigma_i,p_i)-N_2(0,0.5\underline\sigma,p_i)\|_{H^2}\\
  &\leq& C(\|v_i\|_{H^3}+\|\sigma_i-0.5\underline\sigma
  \|_{H^3})\leq C.
\end{eqnarray*}
It follows from \eqref{ALGHk} that
\begin{eqnarray*}
  &&\|N_3(v_1,p_1)-N_3(v_2,p_2)\|_{H^2}\nonumber\\
  &\leq&\|\bar v_1\cdot\nabla_h(p_1-p_2)\|_{H^2}+\|(\bar v_1-\bar v_2)\cdot\nabla_hp_2\|_{H^2}+\gamma\|p_1\text{div}_h(\bar v_1-\bar v_2)\|_{H^2}\\
  &&+\gamma(p_1-p_2)\text{div}_h\bar v_2\|_{H^2}+(\gamma-1)\|\overline{Q(\nabla v_1)}-\overline{Q(\nabla v_2)}\|_{H^2}\\
  &\leq&C(\|v_1\|_{H^2}\|p_1-p_2\|_{H^3}+\|v_1-v_2\|_{H^2}\|p_2\|_{H^3}+ \|p_1\|_{H^2}\| v_1-v_2\|_{H^3}\\
  &&+\|p_1-p_2\|_{H^2}\| v_2\|_{H^3}+\|v_1+v_2\|_{H^3}\|v_1-v_2\|_{H^3})\\
  &\leq&C(\|v_1-v_2\|_{H^3}+\|p_1-p_2\|_{H^2}).
\end{eqnarray*}
Thus, noticing that $N_3(0,p_i)=0$, one gets
\begin{eqnarray*}
  \|N_3(v_i,p_i)\|_{H^2}= \|N_3(v_i,p_i)-N_3(0,p_i)\|_{H^2}\leq C\|v_i\|_{H^3}\leq C.
\end{eqnarray*}
The proof is complete.
\end{proof}

\subsection{Estimates on the linear system}
Let $\mathbb T^3$ be the three dimensional torus. For positive time $T$ denote $Q_T:=\mathbb T^3\times(0,T)$.
Throughout this subsection, the following facts will be used without further mentions
$$
  \|\nabla^2u\|_2^2=\|\Delta u\|_2^2,\quad\|\nabla\Delta u\|_2^2=\|\nabla^3 u\|_2^2,\quad \|\nabla^4u\|_2^2=\|\Delta^2 u\|_2^2,\quad\forall f\in H^4(\mathbb T^3),
$$
which can be verified directly by integrating by parts, with the help of density arguments if necessary.

We start with a spatially periodic initial boundary value problem which, under some symmetry assumptions, is equivalent to the
linear system that we required.

\begin{proposition}
  \label{PROPA2-1}
  Given $u_0\in H^3(\mathbb T^3)$, $f\in L^2(0,T; H^2(\mathbb T^3))$, and $a,b\in L^\infty(Q_T)\cap L^4(0,T; H^2(\mathbb T^3)$, satisfying $a\geq\underline a$ and $b\geq0$ for some positive number $\underline a$. Assume
  that $u\in\mathscr{Z}_T$ is a solution to
  \begin{eqnarray}
    \partial_tu-a\Delta u-b\nabla_h\text{div}_hu=f,\label{A2-1}\\
    u \text{ is  periodic in }x,y,z,\quad u|_{t=0}=u_0.\nonumber
  \end{eqnarray}
  Then, there is a positive constant $C$ depending only on $\underline a$, such that
  \begin{eqnarray*}
    \sup_{0\leq t\leq T}\|u\|_2^2+\int_0^T\|u\|_{H^1}^2dt \leq  C\left(\|u_0\|_2^2+\int_0^T\|f\|_2^2dt\right)e^{C\left(T+\int_0^T\|(\Delta a,\Delta b)\|_2^4dt\right)}, \\
    \sup_{0\leq t\leq T}\|\nabla u\|_2^2+\int_0^T\|\Delta u\|_2^2dt \leq C\left(\|\nabla u_0\|_2^2+\int_0^T\|f\|_2^2dt\right)e^{C \int_0^T\|\Delta b\|_2^4dt},
  \end{eqnarray*}
and
  \begin{align*}
    \sup_{0\leq t\leq T}\|D^k u\|_2^2+\int_0^T\|D^{k+1} u\|_2^2dt
    \leq& C\left(\|D^ku_0\|_2^2+\int_0^T\|D^{k-1} f\|_2^2dt\right)\\
    &\times e^{C \int_0^T\|(\Delta a,\Delta b)\|_2^4dt },\quad k=2,3.
  \end{align*}
Moreover, the following estimates hold
$$
    \int_0^T\|\partial_tu\|_2^2dt \leq  C\left(\|\nabla u_0\|_2^2+\int_0^T\|f\|_2^2dt\right)e^{C \int_0^T\|\Delta b\|_2^4dt}\|(a,b)\|_{L^\infty(Q_T)}^2 ,
$$
and
\begin{align*}
    \int_0^T\|D^{k-1} \partial_tu\|_2^2dt\leq&  C\left(\|D^ku_0\|_2^2+\int_0^T\|D^{k-1} f\|_2^2dt\right)\\
    &\times e^{C \int_0^T\|(\Delta a,\Delta b)\|_2^4dt}\|(a,b)\|_{L^\infty(Q_T)}^2,\quad k=2,3.
\end{align*}
\end{proposition}

\begin{proof}
  Multiplying \eqref{A2-1} with $u$ and integrating by parts, it follows from the Gagliardo-Nirenberg and Young inequalities that
  \begin{eqnarray*}
    &&\frac12\frac{d}{dt}\|u\|_2^2+\int(a|\nabla u|^2+b|\text{div}_hu|^2)dxdydz\\
    &=&-\int(\nabla a\cdot\nabla u\cdot u+\text{div}_hu\nabla_hb\cdot u-f\cdot u)dxdydz\\
    &\leq&C\int[(|\nabla a|+|\nabla b|)|\nabla u||u|+|f||u|]dxdydz\\
    &\leq&C(\|(\nabla a,\nabla b)\|_6\|\nabla u\|_2\|u\|_3+\|f\|_2\|u\|_2)\\
    &\leq&C\left[\|(\Delta a,\Delta b)\|_2\|\nabla u\|_2\|u\|_2^\frac12\left(\|u\|_2+\|\nabla u\|_2\right)^\frac12+\|f\|_2\|u\|_2\right]\\
    &\leq&0.5\underline a\|\nabla u\|_2^2+C\left(\|(\Delta a,\Delta b)\|_2^4+1\right)\|u\|_2^2+C\|f\|_2^2 ,
  \end{eqnarray*}
  and thus
  \begin{equation*}
    \frac{d}{dt}\|u\|_2^2+\underline a\|u\|_{H^1}^2\leq C\left[\left(1+\|(\Delta a, \Delta b)\|_2^4\right)\|u\|_2^2+\|f\|_2^2\right].
  \end{equation*}
  Thanks to this and by the Gr\"onwall inequality, one obtains
  \begin{equation*}
    \sup_{0\leq t\leq T}\|u\|_2^2+\int_0^T\|u\|_{H^1}^2dt\leq C\left(\|u_0\|_2^2+\int_0^T\|f\|_2^2dt\right)e^{C\left(T+\int_0^T\|(\Delta a,\Delta b)\|_2^4dt\right)}.\label{ESTU}
  \end{equation*}

  Multiplying \eqref{A2-1} with $-\Delta u$ and integrating by parts, it follows from the Gagliardo-Nirenberg and Young inequalities that
  \begin{eqnarray*}
    &&\frac12\frac{d}{dt}\|\nabla u\|_2^2+\int a|\Delta u|^2dxdydz\\
    &=&-\int(b\nabla_h\text{div}_hu+f)\cdot\Delta u dxdydz\\
    &=&\int(b\text{div}_hu\Delta\text{div}_hu+\nabla_hb\cdot\Delta u\text{div}_hu-f\cdot\Delta u)dxdydz\\
    &=&-\int (b|\nabla\text{div}_hu|^2+\text{div}_hu\nabla b\cdot\nabla\text{div}_hu-\nabla_hb\cdot\Delta u\text{div}_hu+f\cdot\Delta u)dxdydz\\
    &\leq&C(\|\nabla b\|_6\|\nabla u\|_3\|\nabla^2u\|_2+\|f\|_2\|\Delta u\|_2)\\
    &\leq&C(\|\Delta b\|_2\|\nabla u\|_2^\frac12\|\Delta u\|_2^\frac32+\|f\|_2\|\Delta u\|_2)\\
    &\leq&0.5\underline a\|\Delta u\|_2^2+C\left(\|\Delta b\|_2^4\|\nabla u\|_2^2+\|f\|_2^2\right) ,
  \end{eqnarray*}
  and thus
  \begin{equation*}
    \frac{d}{dt}\|\nabla u\|_2^2+\underline a\int |\Delta u|^2dxdydz\leq C\left(\|\Delta b\|_2^4\|\nabla u\|_2^2+\|f\|_2^2\right).
  \end{equation*}
  Thanks to this and by the Gr\"onwall inequality, one obtains
  \begin{equation}
    \sup_{0\leq t\leq T}\|\nabla u\|_2^2+\int_0^T\|\Delta u\|_2^2dt\leq C\left(\|\nabla u_0\|_2^2+\int_0^T\|f\|_2^2dt\right)e^{C \int_0^T\|\Delta b\|_2^4dt}.\label{EST1U}
  \end{equation}

  Multiplying \eqref{A2-1} with $\Delta^2u$ and integrating by parts, one deduces by the Gagliardo-Nirenberg and Young inequalities that
  \begin{eqnarray*}
    &&\frac12\frac{d}{dt}\|\Delta u\|_2^2+\int a|\nabla\Delta u|^2dxdydz\\
    &=&\int(f+b\nabla_h\text{div}_hu)\cdot\Delta^2udxdyxz-\int\partial_ia\Delta u\cdot\partial_i\Delta udxdydz\\
    &=&\int(b\nabla_h\text{div}_h\Delta u+\Delta b\nabla_h\text{div}_hu+2\partial_ib\nabla_h\text{div}_h\partial_i u)\cdot\Delta udxdydz\\
    &&-\int\nabla f:\nabla\Delta udxdydz-\int\partial_ia\Delta u\cdot\partial_i\Delta udxdydz\\
    &=&-\int (b|\text{div}_h\Delta u|^2+\nabla_hb\text{div}_h\Delta u\cdot\Delta u)dxdyd\\
    &&+\int(\Delta b\nabla_h\text{div}_hu+2\partial_ib\nabla_h\text{div}_h\partial_i u)\cdot\Delta udxdydz\\
    &&-\int\nabla f:\nabla\Delta udxdydz-\int\partial_ia\Delta u\cdot\partial_i\Delta udxdydz\\
    &\leq&C(\|(\nabla a,\nabla b)\|_6\|\nabla^3u\|_2\|\Delta u\|_3+\|\Delta b\|_2\|\nabla^2u\|_4^2+\|\nabla f\|_2\|\nabla\Delta u\|_2)\\
    &\leq&C\left(\|(\Delta a, \Delta b)\|_2\|\nabla^2u\|_2^\frac12\|\nabla^3u\|_2^\frac32+\|\nabla f\|_2\|\nabla\Delta u\|_2\right)\\
    &\leq&0.5\underline a\|\nabla\Delta u\|_2^2+C(\|(\Delta a,\Delta b)\|_2^4\|\Delta u\|_2^2+\|\nabla f\|_2^2) ,
  \end{eqnarray*}
  and thus
  \begin{equation*}
    \frac{d}{dt}\|\Delta u\|_2^2+\underline a\|\nabla\Delta u\|_2^2
    \leq C(\|(\Delta a,\Delta b)\|_2^4\|\Delta u\|_2^2+\|\nabla f\|_2^2).
  \end{equation*}
  Applying the Gr\"owall inequality to the above leads to
  \begin{equation}
  \label{EST2U}
        \sup_{0\leq t\leq T}\|\Delta u\|_2^2+\int_0^T\|\nabla\Delta u\|_2^2dt
        \leq C\left(\|\Delta u_0\|_2^2+\int_0^T\|\nabla f\|_2^2dt\right)e^{C \int_0^T\|(\Delta a,\Delta b)\|_2^4dt }.
  \end{equation}

  Applying $\Delta$ to \eqref{A2-1}, multiplying the resultant with $-\Delta^2u$, and integrating by parts, it follows
  \begin{eqnarray*}
    &&\frac12\frac{d}{dt}\|\nabla\Delta u\|_2^2+\int a|\Delta^2u|^2dxdydz\\
    &=&-\int(\Delta a\Delta u+2\partial_ia\partial_i\Delta u)\cdot\Delta^2udxdydz\\
    &&-\int\Delta(f+b\nabla_h\text{div}_hu)\cdot\Delta^2udxdydz\\
    &=&-\int(\Delta a\Delta u+2\partial_ia\partial_i\Delta u+\Delta f)\cdot\Delta^2udxdydz\\
    &&-\int(\Delta b\nabla_h\text{div}_hu+2\partial_ib\nabla_h\text{div}_h\partial_iu)\cdot\Delta^2udxdydz\\
    &&-\int b\nabla_h\text{div}_h\Delta u\cdot\Delta^2udxdydz,
  \end{eqnarray*}
  from which, noticing that
  \begin{eqnarray*}
    &&-\int b\nabla_h\text{div}_h\Delta u\cdot\Delta^2udxdydz\\
    &=&\int(b\text{div}_h\Delta u\cdot\Delta^2\text{div}_hu+\text{div}_h\Delta u\nabla_hb\cdot\Delta^2u)dxdydz\\
    &=&-\int(b|\nabla\text{div}_h\Delta u|^2+\text{div}_h\Delta u\nabla b\cdot\nabla\Delta\text{div}_hu-\text{div}_h\Delta u\nabla_hb\cdot\Delta^2u)
    dxdydz\\
    &\leq&C\int|\nabla b||\nabla\Delta u||\nabla^4u|dxdydz,
  \end{eqnarray*}
  one obtains by the Gagliardo-Nirenberg and Young inequalities that
  \begin{eqnarray*}
    &&\frac12\frac{d}{dt}\|\nabla\Delta u\|_2^2+\int a|\Delta^2u|^2dxdydz\\
    &\leq&C\int(|\Delta a||\Delta u|+|\nabla a||\nabla\Delta u|+|\Delta f|)|\Delta^2u|dxdydz\\
    &&+C\int(|\Delta b||\nabla^2u|+|\nabla b||\nabla^3u|)|\nabla^4u|dxdydz\\
    &\leq&C(\|(\Delta a, \Delta b)\|_2\|\nabla^2u\|_\infty+\|(\nabla a,\nabla b)\|_6\|\nabla^3u\|_3+\|\Delta f\|_2)\|\nabla^4u\|_2\\
    &\leq&C\left(\|(\Delta a,\Delta b)\|_2\|\nabla\Delta u\|_2^\frac12\|\Delta^2u\|_2^\frac32+\|\Delta f\|_2\|\Delta^2u\|_2\right)\\
    &\leq&0.5\underline a\|\Delta^2u\|_2^2+C\left(\|(\Delta a,\Delta b)\|_2^4\|\nabla\Delta u\|_2^2+\|\Delta f\|_2^2\right).
  \end{eqnarray*}
  Therefore,
  \begin{equation*}
    \frac{d}{dt}\|\nabla\Delta u\|_2^2+\underline a\|\Delta^2u\|_2^2\leq C\left(\|(\Delta a,\Delta b)\|_2^4\|\nabla\Delta u\|_2^2+\|\Delta f\|_2^2\right),
  \end{equation*}
  which by the Gr\"onwall inequality implies
  \begin{equation}
  \label{EST3U}
    \sup_{0\leq t\leq T}\|\nabla\Delta u\|_2^2+\int_0^T\|\Delta^2u\|_2^2dt
    \leq C\left(\|\nabla\Delta u_0\|_2^2+\int_0^T\|\Delta f\|_2^2dt\right)e^{C\int_0^T\|(\Delta a,\Delta b)\|_2^2dt}.
  \end{equation}

  It follows from \eqref{A2-1} that
  \begin{eqnarray}
    \int_0^T\|\partial_tu\|_2^2dt&\leq&C\int_0^T(\|a\Delta u+b\nabla_h\text{div}_hu\|_2^2+\|f\|_2^2)dt\nonumber\\
    &\leq& C\left(\|(a,b)\|_{L^\infty(Q_T)}^2\int_0^T\|\Delta u\|_2^2dt+\int_0^T\|f\|_2^2dt\right). \label{ESTDTU}
  \end{eqnarray}
Using the Gagliardo-Nirenberg and Young inequalities, one obtains from \eqref{A2-1} that
  \begin{eqnarray}
    &&\int_0^T\|\nabla\partial_tu\|_2^2dt\leq C\int_0^T(\|\nabla(a\Delta u+b\nabla_h\text{div}_hu)\|_2^2+\|\nabla f\|_2^2)dt\nonumber\\
    &\leq&C \int_0^T\int[(|a|^2+|b|^2)|\nabla^3u|^2+(|\nabla a|^2+|\nabla b|^2)|\nabla^2u|^2+|\nabla f|^2]dxdydzdt\nonumber\\
    &\leq&C\int_0^T(\|(a,b)\|_\infty^2\|\nabla\Delta u\|_2^2+\|(\nabla a, \nabla b)\|_6^2\|\nabla^2u\|_3^2+\|\nabla f\|_2^2)dt\nonumber\\
    &\leq&C\int_0^T(\|(a,b)\|_\infty^2\|\nabla\Delta u\|_2^2+\|(\Delta a, \Delta b)\|_2^2\|\Delta u\|_2\|\nabla\Delta u\|_2
    +\|\nabla f\|_2^2)dt\nonumber\\
    &\leq&
    C\left(\int_0^T\|(\Delta a,\Delta b)\|_2^4dt\right)^\frac12\left(\sup_{0\leq t\leq T}\|\Delta u\|_2^2+
    \int_0^T\|\nabla\Delta u\|_2^2dt\right)\nonumber\\
    &&+C\left(\|(a,b)\|_{L^\infty(Q_T)}^2\int_0^T\|\nabla\Delta u\|_2^2dt+ \int_0^T\|\nabla f\|_2^2dt\right),\label{ESTDTU1}
    \end{eqnarray}
and
    \begin{eqnarray}
    &&\int_0^T\|\Delta\partial_tu\|_2^2dt\leq C\int_0^T(\|\Delta(a\Delta u+b\nabla_h\text{div}_hu)\|_2^2+\|\Delta f\|_2^2)dt\nonumber\\\nonumber\\
    &=&C\int_0^T\int(|a|^2|\Delta^2u|^2+|b|^2|\nabla^4u|^2+|\nabla a|^2|\nabla^3u|^2\nonumber\\
    &&+|\nabla b|^2|\nabla^3u|^2+|\Delta a|^2|\Delta u|^2+|\Delta b|^2|\nabla^2u|^2+|\Delta f|^2)dxdydzdt\nonumber\\
    &\leq&C\int_0^T(\|(a,b)\|_\infty^2\|\Delta^2u\|_2^2+\|(\nabla a,\nabla b)\|_6^2\|\nabla^3u\|_3^2\nonumber\\
    &&+\|(\Delta a,\Delta b)\|_2^2\|\nabla^2u\|_\infty^2+\|\Delta f\|_2^2)dt\nonumber\\
    &\leq&C\int_0^T(\|(a,b)\|_\infty^2\|\Delta^2u\|_2^2+\|(\Delta a,\Delta b)\|_2^2\|\nabla\Delta u\|_2\|\Delta^2u\|_2+\|\Delta f\|_2^2)dt\nonumber\\
    &\leq&
    C\left(\int_0^T\|(\Delta a,\Delta b)\|_2^4dt\right)^\frac12\left(\sup_{0\leq t\leq T}\|\nabla\Delta u\|_2^2+
    \int_0^T\|\Delta^2 u\|_2^2dt\right)\nonumber\\
    &&+C\left(\|(a,b)\|_{L^\infty(Q_T)}^2\int_0^T\|\Delta^2 u\|_2^2dt+ \int_0^T\|\Delta f\|_2^2dt\right).\label{ESTDTU2}
  \end{eqnarray}
Combining \eqref{ESTDTU}--\eqref{ESTDTU2} and noticing that $\|(a,b)\|_{L^\infty(Q_T)}\geq\underline a$ and $e^z\geq 1+z\geq2\sqrt z$ for $z\geq0$, one obtains the desired estimates for $\partial_tu$
from \eqref{EST1U}--\eqref{EST3U}.
The proof is complete.
\end{proof}

As a direct corollary of Proposition \ref{PROPA2-1}, we have the following corollary.

\begin{corollary}
\label{CORA2-1}
Assume that all the conditions of Proposition \ref{PROPA2-1} hold. Then, the following estimate holds
\begin{eqnarray*}
 \|u\|_{\mathscr{Z}_T}^2
 \leq C\mathcal K_T(a,b)\left(\|u_0\|_{H^3(\mathbb T^3)}^2+\|f\|_{L^2(0,T;H^2(\mathbb T^3))}^2\right),
\end{eqnarray*}
for a positive constant $C$ depending only on $\underline a$, where
\begin{eqnarray}
  \mathcal K_T{(a,b)}:=e^{C\left(T+ \int_0^T\|(\Delta a,\Delta b)\|_2^4dt\right)}\|(a,b)\|_{L^\infty(Q_T)}^2.\label{KT}
\end{eqnarray}
\end{corollary}

\begin{proposition}
  \label{PROPA2-2}
  Given $u_{0i}\in H^3(\mathbb T^3)$, $f_i\in L^2(0,T; H^2(\mathbb T^3))$, and $a_i,b_i\in L^\infty(Q_T)\cap L^4(0,T; H^2(\mathbb T^3))$, satisfying $a_i\geq\underline a$ and $b_i\geq0$ for some positive number $\underline a$, for $i=1,2$. Assume
  that $u_i\in\mathscr{Z}_T$ is a solution to
  \begin{equation*}
  \left\{
  \begin{array}{l}
    \partial_tu_i-a_i\Delta u_i-b_i\nabla_h\text{div}_hu_i=f_i, \\
    u_i \text{ is periodic in }x,y,z,\\
    u_i|_{t=0}=u_{0i} ,
  \end{array}
  \right.
  \end{equation*}
  for $i=1,2$.   Denote
  $$
  (u,a,b,f,u_0)=(u_1-u_2,a_1-a_2,b_1-b_2,f_1-f_2,u_{01}-u_{02}).
  $$
  Then, there is a positive constant $C$ depending only on $\underline a$, such that
\begin{eqnarray*}
\|u\|_{\mathcal Z_T}^2&\leq& C\mathcal K_T(a_1,b_1)\mathcal K_T(a_2,b_2)\left(\|u_{02}\|_{H^3}^2+ \|f_2\|_{L^2(0,T;H^2)}^2\right)\\
&&\times \left(\|(\Delta a,\Delta b)\|_{L^4(0,T;L^2)}^2+\|(a,b)\|_{ L^\infty(Q_T)}^2\right)\\
  &&+C\mathcal K_T(a_1,b_1)\left(\|u_0\|_{H^3}^2+ \|f\|_{L^2(0,T;H^2)}^2\right),
\end{eqnarray*}
where $\mathcal K_T(a_i,b_i)$ is as in \eqref{KT}.
\end{proposition}

\begin{proof}
By Corollary \ref{CORA2-1}, it holds that
\begin{equation}
\|u_i\|_{\mathcal Z_T}^2
 \leq  C\left(\|u_{0i}\|_{H^3}^2+ \|f_i\|_{L^2(0,T;H^2)}^2\right)\mathcal K_T{(a_i,b_i)} , \label{012201}
\end{equation}
 for $i=1,2$.
Note that
\begin{eqnarray*}
  \partial_tu-a_1\Delta u-b_1\nabla_h\text{div}_hu=F,\\
  u|_{t=0}=u_0,
\end{eqnarray*}
where $F:=f+a\Delta u_2+b\nabla_h\text{div}_hu_2.$
Then, by Corollary \ref{CORA2-1}, one gets
\begin{equation}
\|u\|_{\mathcal Z_T}^2
 \leq C\left(\|u_0\|_{H^3}^2+ \|F\|_{L^2(0,T;H^2)}^2\right)\mathcal K_T(a_1,b_1).\label{012202}
\end{equation}
By means of the same arguments as in \eqref{ESTDTU}, \eqref{ESTDTU1}, and \eqref{ESTDTU2}, it follows that
\begin{eqnarray*}
  \|F\|_{L^2(0,T;H^2)}^2 &\leq&
    C\left(\|(\Delta a,\Delta b)\|_{L^4(0,T;L^2)}^2+\|(a,b)\|_{ L^\infty(Q_T)}\right)\\
    &&\times \| u_2\|_{L^\infty(0,T;H^3)\cap L^2(0,T;H^4)}^2 + C\|f\|_{L^2(0,T;H^2)}^2 ,
\end{eqnarray*}
from which, by \eqref{012201}, one gets
\begin{eqnarray*}
  \|F\|_{L^2(0,T;H^2)}^2
  &\leq& C\|f\|_{L^2(0,T;H^2)}^2+ C\left(\|u_{02}\|_{H^3}^2+ \|f_2\|_{L^2(0,T;H^2}^2\right)\mathcal K_T (a_2,b_2)\\
  &&\times \left(\|(\Delta a,\Delta b)\|_{L^4(0,T;L^2)}^2+\|(a,b)\|_{ L^\infty(Q_T)}\right).
\end{eqnarray*}
Substituting this into \eqref{012202} leads to the conclusion.
\end{proof}

\begin{corollary}
  \label{CORA2-2}
  Given $v_{0i}\in H^3(\mathcal O)$, $f_i\in L^2(0,T; H^2(\mathcal O))$, and $\sigma_i\in L^\infty(\mathcal O\times(0,T))\cap L^4(0,T; H^2(\mathcal O))$, satisfying $\sigma\geq0.5\underline\sigma$ for some positive number $\underline\sigma$, and
\begin{equation}\label{BDYIN}
\partial_zv_{0i}|_{z=0,1}=\partial_zf_i|_{z=0,1}=0, \quad \partial_z\sigma_i|_{z=0,1}=0,\quad i=1,2.
\end{equation}
For $i=1,2$, assume that $V_i\in\mathscr{Y}_T$ satisfies
  \begin{equation*}
  \left\{
  \begin{array}{l}
    \partial_tV_i-\mu\sigma_i\Delta V_i-(\mu+\lambda)\sigma_i\nabla_h\text{div}_hV_i=f_i,\\
    V_i \text{ is  periodic in }x,y,\quad \partial_zV_i|_{z=0,1}=0,\\
    V_i|_{t=0}=v_{0i}.
    \end{array}
    \right.
  \end{equation*}
Denote
$$
(V,\sigma,f,v_0)=(V_1-V_2, \sigma_1-\sigma_2, f_1-f_2, v_{01}-v_{02}).
$$
Then, there is a positive constant $C$ depending only on $\mu, \lambda,$ and $\underline\sigma$, such that
\begin{equation*}
\|V_i\|_{\mathscr{Y}_T}^2
 \leq C\left(\|v_{0i}\|_{H^3(\mathcal O)}^2+\|f_i\|_{L^2(0,T;H^2(\mathcal O))}^2\right)\mathcal L_T(\sigma_i),\quad i=1,2,
\end{equation*}
and
\begin{eqnarray*}
\|V\|_{\mathscr Y_T}^2
&\leq& C\mathcal L_T(\sigma_1)\mathcal L_T(\sigma_2)\left(\|v_{02}\|_{H^3}^2+ \|f_2\|_{L^2(0,T;H^2)}^2\right)\\
&&\times \left(\|\Delta\sigma\|_{L^4(0,T;L^2)}^2+\|\sigma\|_{ L^\infty(\mathcal O\times(0,T))}^2\right)\\
  &&+C\mathcal L_T(\sigma_1)\left(\|v_0\|_{H^3}^2+ \|f\|_{L^2(0,T;H^2)}^2\right),
\end{eqnarray*}
where
\begin{eqnarray}
\mathcal L_T(\sigma_i):=Ce^{C\left(T+\int_0^T\|\Delta\sigma_i\|_2^4dt\right)}\|\sigma_i\|_{L^\infty(\mathcal O\times(0,T))}^2. \label{LT}
\end{eqnarray}
\end{corollary}

\begin{proof}
For $i=1,2$, extend $v_{0i}, \sigma_i,$ and $f_i$ periodically and evenly with respect to $z$, and denote, respectively, by $v_{0i}^\text{ext}, \sigma_i^\text{ext},$ and $f_i^\text{ext}$ the corresponding functions after extension. Then, by \eqref{BDYIN}, one can verify that $v_{0i}^\text{ext}\in H^3(\mathbb T^3), \sigma_i^\text{ext}\in L^4(0,T; H^2(\mathbb T^3))\cap L^\infty(Q_T),$ and $f_i^\text{ext}\in L^2(0,T; H^2(\mathbb T^3))$. Denote by
$V_i^\text{ext}$ as the unique solution to
  \begin{eqnarray*}
    \partial_tV_i^\text{ext}-\mu\sigma_i^\text{ext}\Delta V_i^\text{ext}-(\mu+\lambda)\sigma_i^\text{ext}\nabla_h\text{div}_hV_i^\text{ext}=f_i^\text{ext},\\
    V_i^\text{ext} \text{ is  periodic in }x,y,z,\quad     V_i^\text{ext}|_{t=0}=v_{0i}^\text{ext}.
  \end{eqnarray*}
Since $v_{0i}^\text{ext}, \sigma_i^\text{ext},$ and $f_i^\text{ext}$ are even with respect to $z$, one can show that $V_i^\text{ext}(x,y,-z,t)$ satisfies the same system as $V_i^\text{ext}$ and, as a result, by the uniqueness of solutions, $V_i^{\text{ext}}$ is even in $z$. Then, $\partial_zV_i^\text{ext}$ is odd with respect to $z$ which, by the periodicity in $z$, implies $\partial_zV_i^{\text{ext}}|_{z=0,1}=0$. Therefore, $V_i^\text{ext}|_{\mathcal O}=V_i$ by the uniqueness of solutions. Then, the estimates for $V_i$ and $V_1-V_2$ follows from those for $V_i^\text{ext}$ and $V_1^\text{ext}-V_2^\text{ext}$ guaranteed by Corollary \ref{CORA2-1} and Proposition \ref{PROPA2-2}.
\end{proof}

\begin{corollary}
  \label{CORA2-3}
  Given $\Sigma_{0i}\in H^3(\mathcal O)$, $g_i\in L^2(0,T; H^2(\mathcal O))$, and $\sigma_i\in L^\infty(\mathcal O\times(0,T))\cap L^4(0,T; H^2(\mathcal O))$, satisfying $\sigma\geq0.5\underline\sigma$ for some positive number $\underline\sigma$, and
\begin{equation*}
\partial_z\Sigma_{0i}|_{z=0,1}=\partial_zg_i|_{z=0,1}=\partial_z\sigma_i|_{z=0,1}=0,\quad i=1,2.
\end{equation*}
For $i=1,2$, assume that $\Sigma_i\in\mathscr{Y}_T$ satisfies
  \begin{equation*}
    \left\{
    \begin{array}{l}\partial_t\Sigma_i-\nu\sigma_i\partial_z^2\Sigma_i-\epsilon\Delta_h\Sigma_i=g_i,\\
    \Sigma_i \text{ is  periodic in }x,y,\quad \partial_z\Sigma_i|_{z=0,1}=0,\\
    \Sigma_i|_{t=0}=\Sigma_{0i}.\nonumber
    \end{array}
    \right.
  \end{equation*}
Denote
$$
(\Sigma,\sigma,g,\Sigma_0)=(\Sigma_1-\Sigma_2, \sigma_1-\sigma_2, g_1-g_2, \Sigma_{01}-\Sigma_{02}).
$$
Then, there is a positive constant $C$ depending only on $\nu, \epsilon,$ and $\underline\sigma$, such that
  \begin{equation*}
 \|\Sigma_i\|_{\mathscr{Y}_T}^2
 \leq C\mathcal L_T(\sigma_i)\left(\|\Sigma_{0i}\|_{H^3(\mathcal O)}^2+\|g_i\|_{L^2(0,T;H^2(\mathcal O))}^2\right),\quad i=1,2,
\end{equation*}
and
\begin{eqnarray*}
\|\Sigma\|_{\mathscr Y_T}^2
&\leq& C\mathcal L_T(\sigma_1)\mathcal L_T(\sigma_2)\left(\|\Sigma_{02}\|_{H^3}^2+ \|g_2\|_{L^2(0,T;H^2)}^2\right) \\
&&\times \left(\|\Delta\sigma\|_{L^4(0,T;L^2)}^2+\|\sigma\|_{ L^\infty(\mathcal O\times(0,T))}^2\right)\\
  &&+C\mathcal L_T(\sigma_1)\left(\|\Sigma_0\|_{H^3}^2+ \|g\|_{L^2(0,T;H^2)}^2\right),
\end{eqnarray*}
where $\mathcal L_T(\sigma_i)$ is expressed as in \eqref{LT}, but with the constant $C$ there is now depending on $\nu, \epsilon,$ and $\underline\sigma$.
\end{corollary}

\begin{proof}
  This is proved in a similar way as in Corollary \ref{CORA2-2} by going through the arguments in Proposition \ref{PROPA2-1}, Corollary \ref{CORA2-1}, Proposition \ref{PROPA2-2}, and Corollary \ref{CORA2-2}.
\end{proof}

\subsection{Solvability of the nonlinear system}
Based on the results obtained in the previous two subsections, we are now at the position to prove the local existence to the original nonlinear system. We employ the contracting mapping principle.
Let $M_0$ and $T_0$ be two positive constants to be determined later and denote
\begin{align*}
\mathscr B_0:=\Big\{&(v,\sigma,p)\Big|\|(v,\sigma,p)\|_{\mathscr Y_{T_0}\times\mathscr Y_{T_0}\times\mathscr X_{T_0}}\leq M_0, (v,\sigma,p)|_{t=0}=(v_0,\sigma_0,p_0)\Big\}.
\end{align*}
Recalling the definition of the mapping $\mathfrak F$, it suffices to show the existence of fixed point in $\mathscr B_0$ to $\mathfrak F$, which is guaranteed by the following proposition.

\begin{proposition}
  \label{PROPA3-1}
Let $\mathfrak F$ be the mapping defined by \eqref{DEFF}. Then, there are two positive constants $M_0\geq1$ and $T_0\leq1$ depending only on $\|(v_0,\sigma_0)\|_{H^3(\mathcal O)}+\|p_0\|_{H^3(\mathbb T^2)}$, $\underline\sigma, \underline p$, and the parameters $\mu, \lambda, \nu, \gamma, \epsilon$, such that $\mathfrak F$ is a contracting mapping from $\mathscr B_0$ to itself and, thus, there is a unique fixed point $(v_\text{fix}, v_\text{fix}, p_\text{fix})$ to $\mathfrak F$ in $\mathscr B_0$, which by the definition of $\mathfrak F$, is a solution to system (\ref{EQv})--(\ref{EQp}),
  subject to (\ref{BC1})--(\ref{IC}).
\end{proposition}

\begin{proof}
Denote $X_0=\|(u_0, \sigma_0)\|_{H^3(\mathcal O)}+\|p_0\|_{H^3(\mathbb T^2)}$ and let $M_0$ be a constant to be determined later.
Throughout the proof of this proposition, we use $C_{M_0}$ and $C_{X_0}$ to denote general constants depending on $M_0$ and $X_0$, respectively, which may depend also on $\mu, \lambda, \nu, \gamma, \epsilon, \underline\sigma$, and $\underline p$, while we use $C$ to denote a general constant depending only on $\mu, \lambda, \nu, \gamma, \epsilon, \underline\sigma$, and $\underline p$.

By the Sobolev and H\"older inequalities, it holds for any $t\in (0,T_0)$ that
\begin{eqnarray}
  \left\|\sigma-\sigma_0\right\|_\infty&=&\left\|\int_0^t\partial_t\sigma ds\right\|_\infty
  \leq\int_0^{T_0}\|\partial_t\sigma\|_\infty dt\leq C\int_0^{T_0}\|\partial_t\sigma\|_{H^2}dt\nonumber\\
  &\leq& CT_0^\frac12\left(\int_0^{T_0}\|\partial_t\sigma\|_{H^2}^2dt\right)^\frac12
  \leq CM_0T_0^\frac12\leq0.5\underline\sigma,\label{ESTSIGMAINFTY}
\end{eqnarray}
and similarly
\begin{eqnarray*}
  \left\|p-p_0\right\|_\infty
  \leq0.5\underline p,
\end{eqnarray*}
for any $(v,\sigma,p)\in\mathscr B_0$, by choosing $T_0$ sufficiently small depending only on $M_0$, $\underline\sigma$, and $\underline p$. Thanks to this,
recalling that $\sigma_0\geq\underline\sigma$ and $p\geq\underline p$, and by the Sobolev embedding inequality, it follows that
\begin{equation}
  0.5\underline\sigma\leq\sigma\leq0.5\underline\sigma+\|\sigma_0\|_\infty \leq C_{X_0},\quad p\geq0.5\underline p, \quad \forall (v,\sigma,p)\in\mathscr B_0. \label{ESTSIGMA}
\end{equation}
Then, recalling the expression of $\mathcal L_{T_0}(\sigma)$ in \eqref{LT}, one deduces
\begin{eqnarray}
  \mathcal L_{T_0}(\sigma)&\leq&Ce^{C\left(T_0+\int_0^{T_0}\|\Delta\sigma\|_2^4dt\right)}\|\sigma\|_{L^\infty(\mathcal O\times(0,T_0))}\nonumber\\
  &\leq&Ce^{C\left(1+M_0^4T_0\right)}C_{X_0}
  \leq C_{X_0},\quad \forall (v,\sigma,p)\in\mathscr B_0, \label{ESTLT}
\end{eqnarray}
by choosing $T_0$ sufficiently small such that $M_0^4T_0\leq1$.

For any $(v_i, \sigma_i, p_i)\in\mathscr B_0$, $i=1,2$, noticing that $(\sigma_1-\sigma_2)|_{t=0}=0$, it follows in the same way as \eqref{ESTSIGMAINFTY} that
\begin{eqnarray}
  \|\sigma_1-\sigma_2\|_{L^\infty(\mathcal O\times(0,T_0))}&\leq& CT_0^\frac12\left(\int_0^{T_0}\|\partial_t(\sigma_1-\sigma_2)\|_{H^2}^2dt\right)^\frac12\nonumber\\
  &\leq& CT_0^\frac12\|\sigma_1-\sigma_2\|_{\mathscr Y_{T_0}}. \label{ESTSIGMAINFTYDIFF}
\end{eqnarray}
By Proposition \ref{PROPA1-1} and the H\"older inequality, one concludes that
\begin{eqnarray}
     &&\sum_{j=1}^2\|N_j(v_1,\sigma_1,p_1)-N_j(v_2,\sigma_2,p_2)\|_{L^2(0,T_0; H^2(\mathcal O))}\nonumber\\
     &&+\|N_3(v_1,p_1)-N_3(v_2,p_2)\|_{L^2(0,T_0; H^2(\mathbb T^2))}\nonumber\\
     &\leq&T_0^\frac12\sum_{j=1}^2\|N_j(v_1,\sigma_1,p_1)-N_j(v_2,\sigma_2,p_2)\|_{L^\infty(0,T_0; H^2(\mathcal O))}\nonumber\\
     &&+T_0^\frac12\|N_3(v_1,p_1)-N_3(v_2,p_2)\|_{L^\infty(0,T_0; H^2(\mathbb T^2))}\nonumber\\
     &\leq&C_{M_0}T_0^\frac12\left(\|(v_1-v_2,\sigma_1-\sigma_2)\|_{L^\infty(0,T_0;H^3(\mathcal O))}+\|p_1-p_2\|_{L^\infty(0,T_0;H^3(\mathbb T^2))}\right)\nonumber\\
     &\leq&C_{M_0}T_0^\frac12\|(v_1-v_2,\sigma_1-\sigma_2,p_1-p_2)\|_{\mathscr Y_{T_0}\times\mathscr Y_{T_0}\times
     \mathscr X_{T_0}} ,\label{ESTNDIFF}
\end{eqnarray}
   and
   \begin{eqnarray}\label{ESTNSUM}
   &&\sum_{j=1}^2\|N_j(v,\sigma,p)\|_{L^2(0,T_0; H^2(\mathcal O))}+\|N_3(v,p)\|_{L^2(0,T_0; H^2(\mathbb T^2))}\nonumber\\
   &\leq&T_0^\frac12\sum_{j=1}^2\|N_j(v,\sigma,p)\|_{L^\infty(0,T_0; H^2)}+T_0^\frac12\|N_3(v,p)\|_{L^\infty(0,T_0; H^2)}
   \leq C_{M_0}T_0^\frac12,
   \end{eqnarray}
for any $(v,\sigma, p), (v_i, \sigma_i, p_i)\in\mathscr B_0$, $i=1,2$.

For $(v,\sigma, p)\in\mathscr B_0$, denote
$(V,\Sigma,P)=\mathfrak F(v,\sigma,p)$ which, by the definition of $\mathfrak F$, is a solution to
\eqref{AEQV}--\eqref{AEQP}, subject to \eqref{ABC}--\eqref{AIC}.
Recalling \eqref{ESTSIGMA}, \eqref{ESTLT}, and \eqref{ESTNSUM},
it follows from Corollary \ref{CORA2-2}, Corollary \ref{CORA2-3}, and the standard parabolic estimates that
\begin{eqnarray*}
\|\mathfrak F(v,\sigma,p)\|_{\mathscr Y_{T_0}\times\mathscr Y_{T_0}\times\mathscr X_{T_0}}
&=&\|(V,\Sigma,P)\|_{\mathscr Y_{T_0}\times\mathscr Y_{T_0}\times\mathscr X_{T_0}}\nonumber\\
  &\leq& C\mathcal L_{T_0}(\sigma)\left(\|v_0\|_{H^3(\mathcal O)}+\|N_1(v,\sigma,p)\|_{L^2(0,T_0;H^2(\mathcal O))}\right)\nonumber\\
 &&+C\mathcal L_{T_0}(\sigma)\left(\|\sigma_0\|_{H^3(\mathcal O)}+\|N_2(v,\sigma,p)\|_{L^2(0,T_0;H^2(\mathcal O))}\right)\nonumber\\
 &&+ C(\|p_0\|_{H^3(\mathbb T^2)}+\|N_3(v,p)\|_{L^2(0,T_0;H^2(\mathbb T^2))})\nonumber\\
  &\leq& C_{X_0}\left(X_0+C_{M_0}T_0^\frac12\right)\leq C_{X_0},
\end{eqnarray*}
by choosing $T_0$ sufficiently small such that $C_{M_0}T_0^\frac12\leq1$. Thanks to this, by choosing $M_0\geq C_{X_0}$, $\mathfrak F$ is a mapping from $\mathscr B_0$ to itself.

For $(v_i,\sigma_i,p_i)\in\mathscr B_0$, denote $(V_i,\Sigma_i,P_i)=\mathfrak F(v_i,\sigma_i,p_i)$, $i=1,2$, and set   $(V,\Sigma,P)=(V_1-V_2,\Sigma_1-\Sigma_2,P_1-P_2)$. Then,
recalling \eqref{ESTSIGMA}, by Corollary \ref{CORA2-2} and Corollary \ref{CORA2-3}, it follows from \eqref{ESTLT}--\eqref{ESTNSUM} and the H\"older inequality that
\begin{eqnarray*}
  &&\|\mathfrak F(v_1,\sigma_1,p_1)-\mathfrak F(v_2,\sigma_2,p_2)\|_{\mathscr Y_{T_0}\times\mathscr Y_{T_0}\times\mathscr X_{T_0}}\\
  &=&\|V_1-V_2\|_{\mathscr Y_{T_0}}+\|\Sigma_1-\Sigma_2\|_{\mathscr Y_{T_0}}+\|p_1-p_2\|_{\mathscr X_{T_0}}\\
  &\leq&C\mathcal L_{T_0}^\frac12(\sigma_1)\mathcal L_{T_0}^\frac12(\sigma_2)(\|v_0\|_{H^3}+\|N_1(v_2,\sigma_2,p_2)\|_{L^2(0,T_0;H^2)}
  +\|\sigma_0\|_{H^3}\\
  &&+\|N_2(v_2,\sigma_2,p_2)\|_{L^2(0,T_0;H^2)})
  (\|\Delta(\sigma_1-\sigma_2)\|_{L^2(0,T_0;L^2)}+\|\sigma_1-\sigma_2\|_{L^\infty(\mathcal O\times(0,T_0))})\\
  &&+C\mathcal L_{T_0}^\frac12(\sigma_1)\sum_{j=1}^2\|N_j(v_1,\sigma_1,p_1)-N_j(v_2,\sigma_2,p_2)\|_{L^2(0,T_0;H^2)}\\
  &&+C\|N_3(v_1,p_1)-N_3(v_2,p_2)\|_{L^2(0,T_0; H^2)}\\
  &\leq&C_{X_0}(X_0+C_{M_0}T_0^\frac12)T_0^\frac12\|\sigma_1-\sigma_2\|_{\mathscr Y_{T_0}}\\
  &&+C_{X_0}C_{M_0}T_0^\frac12\|(v_1-v_2,\sigma_1-\sigma_2,p_1-p_2)\|_{\mathscr Y_{T_0}\times\mathscr Y_{T_0}\times
     \mathscr X_{T_0}}\\
  &\leq&\frac12\|(v_1-v_2,\sigma_1-\sigma_2,p_1-p_2)\|_{\mathscr Y_{T_0}\times\mathscr Y_{T_0}\times
     \mathscr X_{T_0}},
\end{eqnarray*}
by choosing $T_0$ sufficiently small depending on $X_0$ and $M_0$. Thus, $\mathfrak F$ is a contracting mapping from $\mathscr B_0$ to it self.
\end{proof}

\smallskip
{\bf Acknowledgments.}
R.K.~acknowledges support by Deutsche Forschungsgemeinschaft through Grant CRC 1114 ``Scaling Cascades in Complex Systems'', Project Number 235221301, Project A02 ``Multiscale data and asymptotic model assimilation for atmospheric flows'' and Grant FOR 5528 ``Mathematical Study of Geophysical Flow Models: Analysis and Computation'', Project No.~500072446, Project~2 ``Scale Analysis and Asymptotic Reduced Models for the Atmosphere''.
The work of J.L.~was supported in part by the National Natural Science Foundation of China (Grant No. 12371204) and the Key Project of National Natural Science Foundation of China (Grant No. 12131010). E.S.T.~has benefited from the inspiring environment of the CRC 1114 ``Scaling Cascades in Complex Systems'', Project Number 235221301, Project C09, funded by the Deutsche Forschungsgemeinschaft (DFG). Moreover, this work was also supported in part by the DFG Research Unit FOR 5528 on Geophysical Flows.

\end{document}